\documentclass{article}
\providecommand{\keywords}[1]{\textbf{\textit{Key words.}} #1}

\usepackage{lipsum}
\usepackage{amsfonts}
\usepackage{graphicx}
\usepackage{epstopdf}
\usepackage{standalone}
\usepackage{algcompatible}

\ifpdf
  \DeclareGraphicsExtensions{.eps,.pdf,.png,.jpg}
\else
  \DeclareGraphicsExtensions{.eps}
\fi

\usepackage{amsopn}

\usepackage{epsfig} 
\usepackage{amsmath} 
\usepackage{amssymb}  
\DeclareMathAlphabet{\mathcal}{OMS}{cmsy}{m}{n}
\usepackage[english]{babel}
\usepackage{url}
\usepackage{enumitem}
\usepackage{bbm}
\usepackage{caption}
\usepackage{algorithm}
\usepackage{color}
\usepackage{pgfplots}
\usepackage{siunitx}
\usepackage{tikz}
\usepackage{tkz-euclide}
\usepackage{chngcntr}
\usepackage[titletoc]{appendix}
\usepackage{verbatim}

\definecolor{ao(english)}{rgb}{0.0, 0.5, 0.0}
\usepackage[colorlinks, citecolor = {ao(english)}, linkcolor = {ao(english)}]{hyperref} 
\usepackage{cleveref}
\usepackage{aliascnt}
\usetikzlibrary{calc}
\usepackage{siunitx}
\usepackage{subcaption}
\usepackage{tikz}
\usepackage{pgfplots}
\usetikzlibrary{arrows,shapes,trees,calc,positioning,patterns,decorations.pathmorphing,decorations.markings}
\usetikzlibrary{matrix}
\usepackage{caption}
\usepgfplotslibrary{groupplots}
\pgfplotsset{compat=newest}
\usepackage{amsthm}

\newtheorem{thm}{Theorem}
\crefname{thm}{Theorem}{Theorems}

\newlist{thmenum}{enumerate}{1} 
\setlist[thmenum]{label=\roman*), ref=\thethm~.\roman*.}
\crefalias{thmenumi}{thm}

\newtheorem{prop}{Proposition}
\crefname{prop}{Proposition}{Propositions}
\newtheorem{lem}{Lemma}
\crefname{lem}{Lemma}{Lemmas}

\crefname{cor}{Corollary}{Corollaries}
\theoremstyle{remark}
\newtheorem{rem}{Remark}
\crefname{rem}{Remark}{Remark}

\crefname{ass}{Assumption}{Assumption}
\usepackage{dsfont}
\let\mathbb=\mathds
\usepackage[numbers,sort&compress]{natbib}

\crefname{conj}{Conjecture}{Conjectures}

\theoremstyle{definition}
\newtheorem{defn}{Definition}
\crefname{defn}{Definition}{Definitions}
\newlist{defnenum}{enumerate}{1} 
\setlist[defnenum]{label=\roman*., ref=\thedefn~(\roman*.)}
\crefalias{defnenumi}{defn}

\crefname{prob}{Problem}{Problems}
\crefname{algorithm}{Algorithm}{Algorithms}

\newcommand{\Rmn}{\mathbb{R}^{n \times m}}

\newcommand{\rk}{\textnormal{rank}}

\newcommand{\argmin}{\operatornamewithlimits{argmin}}

\newcommand{\trace}{\textnormal{trace}}
\newcommand{\inter}{\textnormal{int}}

\newcommand{\minmn}{\min \{ m,n \}}

\newcommand{\prox}{\textnormal{prox}}
\newcommand{\proj}{\Pi}

\newcommand{\opts}{\star}

\newcommand{\dom}{\textnormal{dom}}

\newcommand{\epi}{\textnormal{epi}}
\newcommand{\transp}{\mathsf{T}}

\newcommand{\normrgast}[1]{\left\|{} #1 {} \right\|_{g,r\ast}}
\newcommand{\normrg}[1]{\left \| {} #1{}\right \|_{g^D,r}}
\newcommand{\normkyr}[1]{\left\| {} #1{}\right \|_{\ell_1,r}}

\newcommand{\normkyrast}[1]{\left\| {} #1 {}\right \|_{\ell_{\infty},r\ast}}

\newcommand{\normA}[2]{\left\| {} #1 {}\right \|_{ #2}}
\newcommand{\normAast}[2]{\left\| {}#1 {}\right \|_{ #2\ast}}
\newcommand{\funcdom}{\Rmn \to \mathbb{R} \cup \lbrace \infty \rbrace}

\newcommand{\sort}{\textnormal{sort}}

\colorlet{FigColor1}{blue}
\colorlet{FigColor2}{red}
\colorlet{FigColor3}{ao(english)}
\colorlet{FigColor4}{orange}
\pgfplotsset{every axis plot/.append style={line width=1.5pt}}

\crefformat{equation}{\textup{#2(#1)#3}}
\crefrangeformat{equation}{\textup{#3(#1)#4--#5(#2)#6}}
\crefmultiformat{equation}{\textup{#2(#1)#3}}{ and \textup{#2(#1)#3}}
{, \textup{#2(#1)#3}}{, and \textup{#2(#1)#3}}
\crefrangemultiformat{equation}{\textup{#3(#1)#4--#5(#2)#6}}%
{ and \textup{#3(#1)#4--#5(#2)#6}}{, \textup{#3(#1)#4--#5(#2)#6}}{, and \textup{#3(#1)#4--#5(#2)#6}}

\Crefformat{equation}{#2Equation~\textup{(#1)}#3}
\Crefrangeformat{equation}{Equations~\textup{#3(#1)#4--#5(#2)#6}}
\Crefmultiformat{equation}{Equations~\textup{#2(#1)#3}}{ and \textup{#2(#1)#3}}
{, \textup{#2(#1)#3}}{, and \textup{#2(#1)#3}}
\Crefrangemultiformat{equation}{Equations~\textup{#3(#1)#4--#5(#2)#6}}%
{ and \textup{#3(#1)#4--#5(#2)#6}}{, \textup{#3(#1)#4--#5(#2)#6}}{, and \textup{#3(#1)#4--#5(#2)#6}}

\crefdefaultlabelformat{#2\textup{#1}#3}

\usepgfplotslibrary{external} 

\newcommand{\TheTitle}{Efficient Proximal Mapping Computation for Unitarily Invariant Low-Rank Inducing Norms}

\title{{\TheTitle \thanks{This work was completed while both authors were members of the LCCC Linnaeus Center and the eLLIIT Excellence Center at Lund University. It was financially supported by the Swedish Foundation for Strategic Research and the Swedish Research Council through the project 621-2012-5357. The first author is now with the Engineering Department at Cambridge University.
}} }
\author{Christian Grussler \thanks{Department of Engineering, Cambridge University, United Kingdom, {\texttt{christian.grussler@eng.cam.ac.uk}}.} \and Pontus Giselsson \thanks{Department of Automatic Control, Lund University, Sweden, \texttt{pontusg@control.lth.se}.}}

\begin{document}
	
	\maketitle

		Low-rank inducing unitarily invariant norms have been introduced to convexify problems with low-rank/sparsity constraint. They are the convex envelope of a unitary invariant norm and the indicator function of an upper bounding rank constraint. The most well-known member of this family is the so-called nuclear norm. 

To solve optimization problems involving such norms with proximal splitting methods, efficient ways of evaluating the proximal mapping of the low-rank inducing norms are needed. This is known for the nuclear norm, but not for most other members of the low-rank inducing family. This work supplies a framework that reduces the proximal mapping evaluation into a nested binary search, in which each iteration requires the solution of a much simpler problem. This simpler problem can often be solved analytically as it is demonstrated for the so-called low-rank inducing Frobenius and spectral norms. Moreover, the framework allows to compute the proximal mapping of compositions of these norms with increasing convex functions and the projections onto their epigraphs. This has the additional advantage that we can also deal with compositions of increasing convex functions and low-rank inducing norms in proximal splitting methods.

		\vspace*{.25 cm}
		\vspace*{.25 cm}
\section{Introduction}
\subsection*{Background}
Non-convex optimization problems with rank or cardinality constraint appear in many data driven areas such as machine learning, image analysis and multivariate linear regression \cite{izenman1975reduced,antoulas2005approximation,candes2009exact,candes2010matrix,recht2010guaranteed,velu2013multivariate,hastie2015statistical,vidal2016generalized,chandrasekaran2012convex,elden2007matrix} as well as areas within control such as system identification, model reduction, low-order controller design and low-complexity modelling \cite{antoulas1997approximation,fazel2001rank,ankelhed2011design,zoltowski2014sparsity,zare2014low,grussler2016covariance,miller2012identification,hjalmarsson2012ident,liu2013nuclear,liu2010interior,zorzi2015factor,zorzi2016identification,ishteva2013regularized}. Besides the low-rank constraint, these problems are often convex and can be posed as
\begin{equation}
	\begin{aligned}
		& \underset{X}{\textnormal{minimize}}
		& & L(X)\\
		& \textnormal{subject to}
  & & \rk(X) \leq r
	\end{aligned}
	\label{eq:rank_prob_intro}
\end{equation}
where the loss-function $L$ is proper closed and convex. Therefore, one of the most common techniques for solving such problems is to convexify them using regularizers or by taking convex envelopes \cite{fazel2001rank,grussler2016lowrank,grussler2016low,chandrasekaran2012convex}. A promising class of such regularizers and convex envelopes are the so-called unitarily invariant low-rank inducing norms~\cite{grussler2016lowrank}, which are defined for arbitrary unitarily invariant norms $\| \cdot\|_g$ as 
\begin{align}
\|\cdot \|_{g,r\ast} := (\|\cdot\|_g+\chi_{\rk(\cdot)\leq r}(\cdot))^{\ast \ast}
\label{eq:f_reg_intro}
\end{align}
where $\chi_{\rk(\cdot)\leq r}$ is the indicator function for matrices with at most rank $r$ and $(\cdot)^{\ast \ast}$ the biconjugate, which coincides with the convex envelope. If $L = f_0 + f_1(\| \cdot \|_g)$ where $f_0$ and $f_1$ are convex and $f_1$ is increasing, these envelopes have the advantage that a rank-r solution to the convex problem
\begin{equation}
\begin{aligned}
& \underset{X}{\textnormal{minimize}}
& & f_0(X) + f_1(\|X\|_{g,r\ast})
\end{aligned}
\label{eq:full_problem}
\end{equation}
is guaranteed to be a solution to the non-convex problem \cref{eq:rank_prob_intro}. For instance, this may allow us to determine Frobenius norm optimal low-rank approximation with convex constraints by setting $f_1(\|\cdot\|_{\ell_2}) = \|\cdot\|_{\ell_2}^2$,
where $\|\cdot\|_{\ell_2}$ is the Frobenius norm (see~\Cref{sec:example} and \cite{grussler2015optimal} for details). 

\subsection*{Problem}
Although low-rank inducing norms often admit a representation as semi-
definite programs (SDP) (see~\cite{grussler2016lowrank}), proximal splitting algorithms (see~\cite{combettes2011proximal}) are often used for large-scale problems, where standard interior-point method SDP solvers have too costly iterations (see~\cite{peaucelle2002user,toh2004implementation}). 
To apply such methods to \cref{eq:full_problem}, the proximal mapping to $f_1(\|\cdot\|_{g,r\ast})$ is needed, which is the main objective of this work. 

For some $f_1$ the proximal mapping of $f_1(\|\cdot\|_{g,r\ast})$ can be evaluated directly, while for other $f_1$ it may be very involving or even intractable. This can be circumvented by lifting problem  \cref{eq:full_problem} to the epigraph form 
\begin{equation}
\begin{aligned}
& \underset{X,t}{\textnormal{minimize}}
& & f_0(X) + f_1(t) + \chi_{\epi(\|\cdot\|_{g,r\ast})}(X,t),
\end{aligned}
\label{eq:full_problem_epi}
\end{equation}
where $\chi_{\epi(\|\cdot\|_{g,r\ast})}$ is the indicator function of the epigraph to $\|\cdot\|_{g,r\ast}$. Since the proximal mapping of the one dimensional function $f_1$ is fast to evaluate, tractability of the approach relies on projection onto the epigraph being efficient. Thus, we are also interested in the projections onto $\epi(\|\cdot\|_{g,r\ast})$, i.e., the proximal mapping of $\chi_{\epi(\|\cdot\|_{g,r\ast})}$.

\subsection*{Contribution}
In this work, we first introduce a generic search framework for computing a solution to 
\begin{equation}
\min_{Y,w} f(w) + \frac{\gamma}{2} \|Y-Z\|_{\ell_2}^2 + \chi_{\epi(\|\cdot\|_{g^D,r})}(Y,w), \label{eq:dual_prob}
\end{equation}
where $\gamma > 0$ and $Z\in \Rmn$ a fixed, $f$ is proper, closed and convex, and $\normrg{\cdot}$ is the so-called dual norm (see~\Cref{sec:prelim}) to $\normrgast{\cdot}$. Through the Moreau-decomposition, this will allow us to simultaneously treat the proximal mappings of $f_1(\|\cdot\|_{g,r\ast})$, such as $\gamma \|\cdot\|_{g,r\ast}$, $\frac{\gamma}{2}\|\cdot\|_{g,r\ast}^2$, as well as $\chi_{\epi(\|\cdot\|_{g,r\ast})}$. 

The core step of our framework is the reduction of \cref{eq:dual_prob} to a nested binary search algorithm, where each iteration a much simpler problem is needed to be solved. In many cases, this simpler problem can be solved explicitly. This is demonstrated for $\normA{\cdot}{g}$ being the Frobenius norm and the spectral norm. Finally, after computing a singular value decomposition (SVD), our computations only involve singular values and therefore coincide with the computations for vector-valued problems with cardinality constraint.

Note that for $r=1$, all low-rank inducing norms coincide with the well-known nuclear norm (modulus a constant scalar) and efficient algorithms for computing its proximal mapping exist~\cite{parikh2014proximal}. However, for the other members of the low-rank inducing norms such efficient methods are to our best knowledge still unknown. An analytic approach for the low-rank inducing spectral norm has been studied in \cite{wu2014moreau}. Despite the similarity of using the Moreau-decomposition, this approach is of higher computational cost than what is presented here. This is because of our derived binary search rules. Further, \cite{villa2014proximal} proposes a non-analytic approach for an extended class of not necessarily unitarily invariant low-rank inducing norms (see~\cite{laurent2009group}). This approach, however, depends on the complexity and convergence rates of other optimization algorithms. Finally, \cite{eriksson2015k,lai2014efficient,andersson2016convex,larsson2016convex} consider the special case of the squared low-rank inducing Frobenius norm, but the proximal mapping of the non-squared low-rank inducing Frobenius norm as well as the proximal mapping to general $f_1(\|\cdot\|_{g,r\ast})$ are not considered. Interestingly, our framework shows that the computational complexity in case of $\gamma \|\cdot\|_{g,r\ast}$, $\frac{\gamma}{2}\|\cdot\|_{g,r\ast}^2$, and $\chi_{\epi(\|\cdot\|_{g,r\ast})}$ coincide. In particular, since the algorithms in~\cite{eriksson2015k,lai2014efficient} are special cases of our framework, their computational complexity for the squared low-rank inducing Frobenius norm carries directly over to the non-squared case and the epigraph projection. 

\subsection*{Outline}
The paper is organized as follows. We start by introducing some preliminaries on norms and convex optimization. Subsequently, a formal definition of the class of low-rank inducing norms as well as their application to rank constrained optimization problems is outlined. Then we derive our main results, the binary search framework and outline an algorithm for evaluating their epigraph projections. For the low-rank inducing Frobenius and spectral norms, we make these computations explicit and arrive at implementable algorithms for which the computational cost is analyzed. Subsequently, a case study is performed in order to illustrate the performance of our algorithm when solving a problem of form \cref{eq:full_problem} through proximal splitting. Finally, we draw a conclusion and point the reader to our freely available implementations of these algorithms in MATLAB and Python.

\section{Preliminaries}
\label{sec:prelim}
The set of reals is denoted by $\mathbb{R}$, the set of real vectors by $\mathbb{R}^n$,  the set of vectors with nonnegative entries by $\mathbb{R}^n_{\geq 0}$ and the set of real matrices by $\Rmn$. In the remainder of the paper, we assume with out loss of generality that $n \leq m$. The singular valued decomposition of $X \in \Rmn$ is denoted by $X = \sum_i^{n} \sigma_i(X) u_i v_i^\transp$ with non-increasingly ordered singular values $\sigma_1(X) \geq \dots \geq \sigma_{n}(X)$ (counted with multiplicity). The corresponding vector of all singular values is given by
$$\sigma(X):=(\sigma_1(X),\ldots,\sigma_{n}(X)).$$
For all $x=(x_1,\ldots,x_n)\in\mathbb{R}^n$, we define the $\ell_p$ norms by
\begin{align}
\ell_p(x):=\left(\sum_{i=1}^{q} |x_i|^p\right)^{\frac{1}{p}} \quad \text{and} \quad \ell_{\infty}(x):=\max_{i}|x_i|,
\end{align}
where $|\cdot|$ denotes the absolute value. 

A matrix norm $\| \cdot \|: \Rmn\to\mathbb{R}_{\geq 0}$ is called \emph{unitarily invariant} if for all unitary matrices $U \in \mathbb{R}^{n \times n}$ and $V \in \mathbb{R}^{m \times m}$ and all $X \in \Rmn$ it holds that $\|UXV\|= \|X\|$. Equivalently, unitary invariance can be characterized by \emph{symmetric gauge functions} (see~e.g.~\cite[Theorem~7.4.7.2]{horn2012matrix}):
			\begin{defn}
	\label{def:gauge}
	A function $g: \mathbb{R}^n \to \mathbb{R}_{\geq 0}$ is a symmetric gauge function if
\begin{defnenum}
		\item $g$ is a~norm.
		\item $\forall x \in \mathbb{R}^{n}: g(|x|) = g(x)$, where $|x|$ denotes the element-wise absolute value. \label{def:gauge:absolute}
		\item $g(Px) = g(x)$ for all permutation matrices $P\in\mathbb{R}^{n\times n}$ and all $x\in\mathbb{R}^n$.
	\end{defnenum}
\end{defn}

\begin{prop}
	\label{thm:symmgauge}
	The norm $\| \cdot \|:\Rmn\to\mathbb{R}_{\geq 0}$ is unitarily invariant if and only if $$\| \cdot \| = g(\sigma_1(\cdot),\dots,\sigma_{n }(\cdot))$$ where $g$ is a symmetric gauge function.
\end{prop}
 Throughout this work, we use the notation $\|X\|_g := g(\sigma(X))$. For $X,Y \in \Rmn$ the \emph{Frobenius inner product} is defined as
 \begin{equation*}
 \langle X , Y \rangle := \sum_{i=1}^{m}\sum_{j=n}^{n} x_{ij} y_{ij} = \trace(X^\transp Y)
 \end{equation*}
 with \emph{Frobenius~norm}
  \begin{equation*}
	\|X\|_{\ell_2} := \ell_2(\sigma(X)) = \sqrt{\langle X , X \rangle }=\sqrt{\sum_{i=1}^{n}\sum_{j=1}^{m} x_{ij}^2}.
\end{equation*} Moreover, the \emph{nuclear norm} and the \emph{spectral norm} are given by
\begin{equation*}
	\|\cdot \|_{\ell_1}:=\ell_{1}(\sigma(\cdot)) \quad \text{and} \quad \|\cdot\|_{\ell_\infty}:=\ell_{\infty}(\sigma(\cdot))=\sigma_1(\cdot).
\end{equation*}
The \emph{dual norm} to $\|\cdot \|_g$ is defined as
	\begin{align}
	\|\cdot \|_{g^D} &:= \max_{\|X\|_g \leq 1} \langle \cdot , X \rangle= g^D(\sigma_1(\cdot ),\dots,\sigma_n(\cdot)).
	\label{eq:dual_uni_norm}
\end{align} 
In particular, this means that dual norms inherit the unitary invariance as well as the duality relationship for $\ell_p$ norms, i.e. 
\begin{equation*}
g = \ell_p \quad \Longrightarrow \quad g^D = \ell_q
\end{equation*} 
with $p,q \in [1,\infty ]$ satisfying $\frac{1}{p} + \frac{1}{q} = 1$ (see~e.g.~\cite{luenberger1968optimization}). For example, the Frobenius norm is self-dual, i.e. $g = g^D = \ell_2$ and the dual norm to the spectral norm is the nuclear norm, i.e. $g = \ell_\infty$ with $g^D = \ell_1$. 

Furthermore, in this work the following truncated dual gauge functions will play a key role. To this end, let us define the truncation operator $T: \mathbb{R}^{n} \to \mathbb{R}^{r-t+1}$  for all $1 \leq r \leq n$ and $(t,s) \in \{1,\dots,r\} \times \{0,\dots,n-r\}$ as
\begin{equation}
(Tx)_i := \begin{cases}
\sort(x)_i, &\text{if } 1\leq i \leq r-t, \\
\dfrac{\sum_{i = r-t+1}^{r+s} \sort(x)_i}{\sqrt{t+s}},	&\text{if } i = r-t+1,
\end{cases}	\label{eq:transform}
\end{equation} 
where $\sort: \mathbb{R}^n \to \mathbb{R}^n$ denotes the sorting in descending order, and the corresponding \emph{truncated gauge function of} of $g^D$ as
\begin{equation*}
g^D_{r,s,t}(x) := g^D\bigg((Tx)_1,\dots,(Tx)_{r-t},\underbrace{(Tx)_{r-t+1},\dots,(Tx)_{r-t+1}}_{t~\textnormal{times}},0,\dots,0\bigg)
\end{equation*}
for all $x \in \mathbb{R}^n$. For the special case $(t,s) = (1,0)$, we simply write $g^D_r$. Note that $g^D_{r,s,t}$ is indeed a gauge function with dual gauge function \cite[Lemma~2.2.2]{hiriart1996convex2})
\begin{equation*}
	g_{r,s,t}(x) := g((Tx)_1,\dots,(Tx)_{r-t},\underbrace{\tfrac{{(Tx)}_{r-t+1}(s+t)}{t},\dots,\tfrac{(Tx)_{r-t+1}(s+t)}{t}}_{t~\textnormal{times}},0,\dots,0).
\end{equation*}

For the convince of the reader, we review next some elementary definitions and results from convex optimization. For $f: \funcdom$ we define the following set notions:
\begin{itemize}
	\item \emph{effective domain}: $\dom (f) :=  \lbrace X \in \Rmn: f(X) < \infty \rbrace.$
	\item \emph{epigraph}: $\epi(f) := \lbrace (X,t) : f(X) \leq t, X \in \dom (f), ~\ t \in \mathbb{R} \rbrace.$
	\item \emph{subdifferential in $X \in \dom (f)$}: $$\partial f(X) := \lbrace G \in \Rmn: \langle G, Y-X \rangle \leq f(Y)-f(X) ~\ \text{for all} ~\ Y \in \dom (f) \rbrace.$$
\end{itemize}
In particular, by \cite[Exampel VI.3.1]{hiriart2013convex}
\begin{align}
\label{eq:sub_norm}
\partial \|X\|_g = \{ G \in \Rmn: \langle G, X\rangle = \|X\|_g, \ \|G\|_{g^D} = 1 \}.
\end{align} 
Further, $f$ is said to be: 
			 \begin{itemize}
			 	\item \emph{proper} if $\dom(f) \neq \emptyset$.
			 	\item \emph{closed} if $\epi(f)$ is a closed set.
			 \end{itemize}
		 The \emph{conjugate (dual) function} $f^\ast$ of $f$ is defined as $$f^\ast(Y) :=   \sup_{X \in \Rmn} \left[ \langle X, Y \rangle - f(X) \right]$$
		 for all $Y \in \Rmn$. The function $f^{\ast \ast} := (f^\ast)^\ast$ is called the \emph{biconjugate function} or \emph{convex envelope} of $f$. 
For $f:\mathbb{R} \to\mathbb{R}\cup\lbrace\infty\rbrace$, we say that $f$ \emph{increasing} if $$x \leq y \ \Rightarrow \ f(x) \leq f(y) \ \text{ for all } \ x,y \in \dom(f)$$ and if there exist $x,y\in\mathbb{R}$ such that $x<y$ and $f(x)<f(y)$. Moreover, its monotone \emph{monotone conjugate} is defined as \cite{rockafellar1970convex} $$f^+(y) := \sup_{x \geq 0} \left[ xy - f(x) \right] \ \text{ for all } y\in\mathbb{R}.$$ 		 
The indicator function of a set $\mathcal{S} \subset \Rmn$ is defined as
\begin{equation*}
	\chi_{\mathcal{S}}(X) := \begin{cases} 0 & \textnormal{if }X \in \mathcal{S}, \\
		\infty & \textnormal{if } X \notin \mathcal{S}. \end{cases}
\end{equation*}
We also use this notation for the indicator function of the set of matrices with at most rank $r$, i.e. $\chi_{\rk(\cdot)\leq r}$. For any $Z \in  \Rmn$, the \emph{proximal mapping} of a closed, proper and convex function $f: \funcdom$ is defined as
\begin{align}
\prox_{\gamma f}(Z) := \argmin_X \left(f(X)+\dfrac{1}{2\gamma} \|X-Z\|^2_{\ell_2} \right). \label{eq:prox_op}
\end{align}
In particular, $\prox_{\gamma \chi_{\mathcal{C}}}(Z)$ coincides with the unique Euclidean projection
\begin{align*}
	\proj_{\mathcal{C}}(Z) := 	\argmin_{X \in \mathcal{C}} \|X-Z\|_{\ell_2}
\end{align*}
onto $\mathcal{C}$ for any closed, non-empty set $\mathcal{C} \subset \Rmn$. Moreover, by the \emph{extended Moreau decomposition} it holds for all $f: \funcdom$, $Z \in \Rmn$ and $\gamma > 0$ that (see~\cite[Theorem~6.29]{bauschke2011convex})
\begin{align}
\prox_{\gamma f}(Z) = Z - \gamma \prox_{\gamma^{-1} f^\ast}(\gamma^{-1}Z). \label{eq:moreau_decomp}
\end{align}


\section{Low-Rank Inducing Norms}
\label{sec:norms}
This section introduces the family of unitarily invariant \emph{low-rank inducing norms}, which has been discussed in~\cite{grussler2016lowrank}. Besides recapping some elementary properties, this section briefly motivates the usefulness of these norms as convex envelopes or additive regularizers in optimization problems to promote low-rank solutions.

Low-rank inducing norms are defined as the dual norm of a \emph{rank constrained dual norm}
\begin{equation}
\| Y \|_{g^D,r} :=  \max_{\stackrel{\rk(X) \leq r}{\|X\|_g \leq 1}} \langle X, Y \rangle. \label{eq:def_rank_ind_norm_dual}
\end{equation}
This means that the \emph{low-rank inducing norms} corresponding to $\|\cdot\|_g$ are given by
\begin{align}
\|X\|_{g,r*} := \max_{\|Y\|_{g^D,r}\leq 1}\langle Y,X\rangle.
\label{eq:def_rank_ind_norm}
\end{align}
For {{$r=n$}}, the rank constraint in \cref{eq:def_rank_ind_norm_dual} is redundant and $\|\cdot\|_g \equiv \| \cdot \|_{g,r*}$. 
Some important properties of these norms are summarized next \cite{grussler2016lowrank}.
	\begin{lem}
		\label{lem:norm}
		Let $X,Y \in \Rmn$, $r \in \mathbb{N}$ be such that $1\leq r \leq n$, and $g: \mathbb{R}^n \to \mathbb{R}_{\geq 0}$ be a symmetric gauge function. Then $\|\cdot\|_{g^D,r}$ is a unitarily invariant norm with
		\begin{equation}
\| Y \|_{g^D,r} = g^D_r(\sigma(Y)).
\label{eq:firstass}
		\end{equation} 
Its dual~norm $\|\cdot\|_{g,r\ast}$ satisfies
		\begin{align}
& \|\cdot\|_{g,r\ast} = (\| \cdot \|_g +  \chi_{\rk(\cdot)\leq r}(\cdot))^{\ast \ast}. \label{eq:convexhull}
		\end{align}
	\end{lem} 
In this work, we especially consider the so-called \emph{low-rank inducing Frobenius norm} 
	\begin{align*}
	\|X\|_{\ell_2,r*:}=\max_{\|Y\|_{\ell_2,r} \leq 1} \langle Y,X \rangle
\end{align*}
and the \emph{low-rank inducing spectral norm} 
\begin{align*}
	\normkyrast{X}:=\max_{\normkyr{Y}\leq 1}\langle Y,X\rangle.
\end{align*}
The following motivates the main interest in low-rank inducing norms (see \cite{grussler2016lowrank,grussler2017local,grussler2016low} for details).
	\begin{prop}
	\label{prop:opt_reg}
	Assume that $f_0: \funcdom$ is a proper closed convex function, and that $r \in \mathbb{N}$ is such that $1 \leq r \leq \minmn$. Let $f_1: \mathbb{R}_{\geq 0} \to \mathbb{R} \cup \{ \infty \}$ be an increasing, proper closed convex function, and let $\theta >0$. Then
	\begin{equation}
			f_1(\normrgast{\cdot})^\ast = f_1^+(\normrg{\cdot}) \label{eq:conjugate_norm}
	\end{equation}
	and
	\begin{align}
		\inf_{\stackrel{X \in \Rmn}{\rk(X) \leq r}} \left[ f_0(X) + \theta f_1(\|X\|_g) \right] & \geq -\inf_{D \in \Rmn} \left[f_0^\ast(D) + \theta f^{+}(\theta^{-1}\|D\|_{g^D,r})  \right] \label{eq:fenchel_dual_main}\\
		 &= \inf_{X \in \Rmn} \left[f_0(X) + \theta f_1(\|X\|_{g,r\ast}) \right] \label{eq:fenchel_inequ_main}.
	\end{align}
	If $X^\opts$ solves \cref{eq:fenchel_inequ_main} such that $\rk(X^\opts) \leq r$, then equality holds, and $X^\opts$ is also a solution to the problem on the left of \cref{eq:fenchel_dual_main}.
\end{prop}
In other words, \cref{prop:opt_reg} shows that low-rank inducing norms can be used both as additive regularizers and direct convex envelopes to find (approximate) solutions to 
\begin{equation}
\begin{aligned}
& \underset{X}{\textnormal{minimize}}
& & f(X)\\
& \textnormal{subject to}
& & \rk(X) \leq r.
\end{aligned}
\label{eq:rank_prob_norm}
\end{equation}
For regularization, \cite{fazel2001rank,tibshirani1996regression}, we set $f_0 = f$ and choose a suitable $f_1$ and $\theta$ to find an approximate solution. In the second case, when $f$ can be split into $f = f_0 + f_1(\|\cdot\|_g)$ as in \cref{prop:opt_reg}, then
	\begin{align}
	\min_{X \in \Rmn} \left[f_0(X) + f_1(\|X\|_{g,r\ast}) \right] \label{eq:main_convex_prob}
\end{align}
may return an (exact) solution to \cref{eq:rank_prob_norm}.

\section{Proximal Mappings}
\label{sec:main}
For problems of small size, it is often convenient to solve \cref{eq:rank_prob_norm} through semi-definite programming (SDP). However, conventional SDP solvers are typically based on interior-point methods (see~\cite{peaucelle2002user,toh2004implementation}) with iteration cost that grows unfavorably with the problem dimension. For large-scale problems, proximal splitting methods can be used (see~\cite{bauschke2011convex,combettes2011proximal}). 

To efficiently solve \cref{eq:rank_prob_norm}, proximal splitting methods require efficient computation of the proximal mapping of $f_1(\|\cdot\|_{g,r\ast})$. In this section, we present our main results on developing a nested binary search framework (see~\cref{thm:main,alg:proj_epi}) for computing this proximal mapping for simple $f_1$ efficiently. Explicit and implementable steps for these computations will be shown for the simple, but most frequently appearing cases \cite{argyriou2012sparse,McDonaldPS15,grussler2016lowrank,grussler2016low}
\begin{itemize}
	\item $f_1 = (\cdot)$ and $f_1 = (\cdot)^2$
	\item  $g = \ell_2$ and $g = \ell_{\infty}$
\end{itemize}
In \Cref{subsec:compl}, the computational complexity of our generic algorithm as well as these particular cases is derived.

In cases, where $f_1$ is not simple, \cref{eq:rank_prob_norm} can be rewritten as
\begin{equation}
\min_{t \in \mathbb{R},~X \in \Rmn} f_0(X) + f_1(t) + \chi_{\epi(\|\cdot\|_{g,r\ast})}(X,t),
\label{eq:main_full_problem_epi}
\end{equation}
where $\chi_{\epi(\|\cdot\|_{g,r\ast})}$ is the indicator function of the epigraph to $\|\cdot\|_{g,r\ast}$. Then a consensus formulation for proximal splitting methods (see~\cite{combettes2011proximal}) requires an evaluation of the  proximal mappings for $f_1$ and $\chi_{\epi(\|\cdot\|_{g,r\ast})}$. Since $f_1$ is one-dimensional, convex, proper and increasing, its proximal mapping is fast to evaluate. We will see as part of our complexity analysis in \Cref{subsec:compl} that computing $\prox_{\chi_{\epi(\|\cdot\|_{g,r\ast})}}$ has the same cost as the cases $f_1 = (\cdot)$ and $f_1 (\cdot)^2$. 

Finally note that in contrast to $\normA{\cdot}{g,r\ast}$, its dual norm $\normA{\cdot}{g^D,r}$ is explicitly known by its definition \cref{eq:firstass}. Therefore, we derive our search framework for
\begin{align}
\prox_{\gamma^{-1}f_1^+(\|\cdot\|_{g^D,r})}(Z) \quad \text{and} \quad \proj_{-\epi(\|\cdot\|_{g^D,r})}(Z,z_v), \label{eq:proxes}
\end{align}
with
$$-\epi(\normrg{\cdot}) =\{(Y,-w) : \normrg{Y}\leq w\}$$
which by \cref{eq:moreau_decomp} and \cref{eq:conjugate_norm} yields
\begin{subequations}
\begin{align}
\prox_{\gamma f_1(\normrgast{\cdot})}(Z) &= Z - \gamma \prox_{\gamma^{-1} f_1^+(\|\cdot\|_{g^D,r})}(\gamma^{-1}Z) \label{eq:Moreau_explicit}\\
\prox_{\chi_{\epi(\|\cdot\|_{g,r\ast})}}(Z,z_v) &= (Z,z_v) - \proj_{-\epi(\|\cdot\|_{g^D,r})}(Z,z_v)\label{eq:proj_epi}.
\end{align} 
\end{subequations}
\subsection{Search framework}
\label{sec:prox_derv}
Next, we present our main result, which shows that \cref{eq:proxes}, and hence \cref{eq:Moreau_explicit,eq:proj_epi}, can be computed by a nested parameter search. Since the computations of \cref{eq:proxes} can be unified as
	\begin{equation}
	\label{opt:dual_prox}
	\begin{aligned}
	& \underset{Y,w}{\textnormal{minimize}}
	& & f(w) + \frac{\gamma}{2} \|Y-Z\|_{\ell_2}^2\\
	& \textnormal{subject to}
	& & w \geq  \normrg{Y}, \ Y \in \Rmn
	\end{aligned}
	\end{equation}	
	where $f$ is closed, proper and convex, our results are stated for all such problems. In particular, for the cases
	\begin{equation}
	\text{(i)} \ \proj_{\epi(\normrgast{\cdot})}(Z,z_v) \qquad \text{(ii)} \ \prox_{\frac{\gamma}{2}{\normAast{\cdot}{g,r}^2}}(Z) \qquad \text{(iii)} \ \prox_{\gamma {\normAast{\cdot}{g,r}}}(Z) \label{eq:important_cases}
	\end{equation}
	we choose by \cref{eq:moreau_decomp} and \cref{eq:conjugate_norm} 
	\begin{enumerate}[label=(\roman*)]
		\item $f(w) = \frac{1}{2}(w+z_v)^2$ and $\gamma = 1$ \label{case: epi}
		\item $f(w) = \begin{cases} \frac{1}{2}w^2 & w \geq 0\\
		0 & w < 0
		\end{cases}$. \label{case: square}
		\item $f(w) = \chi_{[0,\gamma]}(w)$ \label{case: nonsquare}
	\end{enumerate}
	such that the corresponding solution $(Y^\opts,w^\opts)$ to \cref{opt:dual_prox} yield
		\begin{enumerate}[label=(\roman*)]
		\item $(Y^\opts,-w^\opts) = \proj_{-\epi(\|\cdot\|_{g^D,r})}(Z,z_v)$
		\item $Y^\opts = \prox_{\frac{\gamma}{2}{\normAast{\cdot}{g^D,r}^2}}(Z)$
		\item $Y^\opts = \prox_{\chi_{\|\cdot \|_{g^D} \leq \gamma }(Z)}$ 
	\end{enumerate}
where $\chi_{\|\cdot \|_{g^D} \leq \gamma }$ is the indicator function of the set $\{X: \|X\|_{g^D} \leq \gamma \}$.
\begin{thm}
	\label{thm:main}
	Let $Z = \sum_{i = 1}^n \sigma_i(Z) u_i v_i^\transp \in \Rmn$, $\gamma > 0$, $1\leq r\leq n$, $g: \mathbb{R}^n \to \mathbb{R}$ be a gauge function, and $f:\mathbb{R} \to \mathbb{R} $ be a proper, closed and convex. For each $(t,s) \in \{1, \dots, r\} \times \{0,\dots, n-r\}$ let
	$(y^{(t,s)},w^{(t,s)}) \in \mathbb{R}^{n+1}$ be defined as 
	\begin{subequations}
	\begin{align}
	y^{(t,s)}_i &:= \begin{cases}
	\tilde{y}_i,  &\text{if } 1\leq i \leq r-t,\\
	\dfrac{\tilde{y}_i}{\sqrt{t+s}}, 	&\text{if }  r-t+1 \leq i \leq r+s, \\
	\sigma_i(Z)	&\text{if }  i \geq r+s+1.\\
	\end{cases} \label{eq:tilde_var_y_thm}\\
	w^{(t,s)} &:= \tilde{w} \label{eq:tilde_w_thm}
	\end{align}
	\end{subequations}
	where $(\tilde{y},\tilde{w}) \in \mathbb{R}^{r-t+2}$ fulfills one of the following cases
	\begin{subequations}
		\label{eq:cases_thm}
		\begin{align}
		&\text{Case 1: }\tilde{y} = \tilde{z}, \ \tilde{w}= \argmin_w f, \ \tilde{w} \geq g^D_{r,s,t}(\tilde{z})\label{test:in_epi_thm} \tag{C1}\\
		& \text{Case 2: }  (\tilde{y},\tilde{w}) = 0 ~\ \Longleftrightarrow ~\ g_{r,s,t}(\tilde{z}) \leq  \frac{\mu}{\gamma} \text{ and } \mu \in \partial  f(0) \label{test:zero_sol_thm} \tag{C2}\\
		& \text{Case 3: } \frac{\gamma}{\mu} (\tilde{z}-\tilde{y}) \in \partial {g^D_{r,s,t}}(\tilde{y}),\ \mu \in \partial f(\tilde{w})\cap \mathbb{R}_{\geq 0}, \ \tilde{w} = g^D_{r,s,t}(\tilde{y})\tag{C3} \label{test:boundary_epi_thm} 
		\end{align}
	\end{subequations}
	and $\tilde{z} := T \sigma(Z)$ is given by \cref{eq:transform}. Then $(Y^\opts,w^\opts) = (\sum_{i = 1}^n y_i^{(t^\opts,s^\opts)}u_i v_i^\transp,w^{(t^\opts,s^\opts)})$ is the solution to \cref{opt:dual_prox}, where 
	\begin{subequations}
	\begin{align}
	t^\opts &:= \min \left \lbrace \lbrace t: y^{(t,s_t^\opts)}_{r-t} > y_{r-t+1}^{(t,s_t^\opts)} \rbrace \cup \lbrace r \rbrace \right \rbrace \label{eq: t_opt}\\
	s_t^\opts &:= \min \left \lbrace \lbrace s: y^{(t,s)}_{r+s} > y_{r+s+1}^{(t,s)} \rbrace \cup \lbrace n-r \rbrace \right \rbrace \label{eq:s_t_opt}\\
	s^\opts &:= s_{t^\opts}^\opts \notag
	\end{align} 
	\end{subequations}
	In particular, $(t^\opts,s^\opts)$ can be found by a nested search over $t$ and $s$ with the following rules for increasing/decreasing $t$ and $s$:
		\begin{enumerate}[label=\Roman*.]
		\item $y^{(t,s_t^\opts)}_{r-t} \geq y_{r-t+1}^{(t,s_t^\opts)}$ for all $t \geq t^\opts$.
		\item $y^{(t,s_t^\opts)}_{r-t} \leq y_{r-t+1}^{(t,s_t^\opts)}$ for all $t < t^\opts$. 
		\item If $t < t^\opts$ and $y^{(t,s_t^\opts)}_{r-t} = y_{r-t+1}^{(t,s_t^\opts)}$ then $\left( y^{(t,s_t^\opts)},w^{(t,s_t^\opts)} \right)  = \left( y^{((t,s_t^\opts)},w^{(t,s_t^\opts)} \right) $.
		\item $y^{(t,s)}_{r+s} \geq y_{r+s+1}^{(t,s)}$ for all $s \geq s_t^\opts$.
		\item $y^{(t,s)}_{r+s} \leq y_{r+s+1}^{(t,s)}$ for all $s < s_t^\opts$.
		\item If $s < s_t^\opts$ and $y^{(t,s)}_{r+s} = y_{r+s+1}^{(t,s)}$ then $\left( y^{(t,s)},w^{(t,s)} \right)  = \left( y^{(t,s_t^\opts)},w^{(t,s_t^\opts)} \right)$.
	\end{enumerate}
\end{thm}
A proof of \cref{thm:main} is given in \Cref{sec:proof_main} and a binary search implementation for determining $\proj_{-\epi(\|\cdot\|_{g^D,r})}(Z,z_v)$ is outlined in \Cref{alg:proj_epi}. 

Note that \cref{thm:main} reduces the problem of solving \cref{opt:dual_prox} to the tractability of \cref{test:in_epi_thm,test:zero_sol_thm,test:boundary_epi_thm}.
In the following, \cref{test:in_epi_thm,test:zero_sol_thm,test:boundary_epi_thm} are made explicit for the cases \cref{eq:important_cases} and the low-rank inducing Frobenius and spectral norms, i.e., $g = \ell_2$ and $g = \ell_\infty$. More generally, we will determine \cref{test:in_epi_thm,test:zero_sol_thm,test:boundary_epi_thm}
for $g = \gamma \ell_2$ and $g = \tau \ell_\infty$ for all $\tau > 0$, because this will allow us to handle the first two cases
simultaneously through the identity
\begin{equation*}
\prox_{\frac{\tau^2}{2}{\normAast{\cdot}{g,r}^2}}(Z) 
= Z - \prox_{\frac{1}{2}{\normA{\cdot}{\frac{g^D}{\tau},r}^2}}(Z) = Z- \proj_{Y}\bigg(\proj_{-\epi(\normA{\cdot}{\frac{g^D}{\tau},r})}(Z,0)\bigg)
\end{equation*}
where $\proj_{Y} (Y,w) := Y$ and $\tau > 0$. Further, we will see that it is easy to adjust these computations for the third case, 
because $\prox_{\tau {\normAast{\cdot}{g,r}}}(Z) = \prox_{{\normAast{\cdot}{\tau g,r}}}(Z)$.

\begin{algorithm}[t]
	\caption{Binary search for determining 
		$\proj_{-(\epi(\|\cdot\|_{g^D,r})}(Z,z_v)$}
	\begin{algorithmic}[1]
		\STATE \textbf{Input:}	Let $Z \in \Rmn$, $z_v \in \mathbb{R}$ and $r \in \{1,\dots,n\}$.
		\vspace*{0.1cm}
		\STATE {Let $Z = \sum_{i=1}^n \sigma_i(Z)u_i v_i^\transp$ be an SVD of $Z$ and $z=\sigma(Z)$.}
		\STATEx{ \emph{//Let $f = \frac{1}{2}(w+z_v)^2$ and $\gamma =1$}}
		\STATEx \emph{//Find $(t^\opts,s^\opts)$ in \cref{thm:main} through binary search over $(t,s)$}  
		\STATE{Set $t_{\min} = 1$, $t_{\max} = r$, and $t = \lfloor \frac{t_{\min} + t_{\max}}{2} \rfloor$} 
		\STATEx \emph{//Binary search over $t$ to find $t^\opts$}  
		\WHILE{$t_{\min} \neq t_{\max}$}		
		\STATE{Set $s_{\min} = 0$, $s_{\max} = n-r$, and $s = \lfloor \frac{s_{\min} + s_{\max}}{2} \rfloor$} 
		\STATEx \emph{//Binary search over $s$ to find $s_t^\opts$}
		\WHILE{$s_{\min} \neq s_{\max}$}
		\STATE{Determine $(y^{(t,s)}_{r+s},y^{(t,s)}_{r+s+1})$ in \cref{eq:tilde_var_y_thm,eq:tilde_w_thm}}
		\IF{$y^{(t,s)}_{r+s} < y_{r+s+1}^{(t,s)}$}
		\STATE{$s_{\min} = s+1$}
		\ELSE
		\STATE{$s_{\max} = s$}
		\ENDIF
		\ENDWHILE
		\STATE{Set $s_t^\opts = s_{\min}$}
		\STATE{Determine $(y^{(t,s_t^\opts)}_{r-t},y^{(t,s_t^\opts)}_{r-t+1})$ in \cref{eq:tilde_var_y_thm,eq:tilde_w_thm}}
		\IF{$y^{(t,s)}_{r-t} < y_{r-t+1}^{(t,s)}$}
		\STATE{$t_{\min} = t+1$}
		\ELSE
		\STATE{$t_{\max} = t$}
		\ENDIF
		\ENDWHILE
		\STATE{Set $t^\opts = t_{\min}$ and p.r.n. binary search for $s^\opts = s_{t^\opts}^\opts$}
		\STATE{\textbf{Output: }$(Y^\opts,w^\opts) = (\sum_{i=1}^ny_i^{(t^\opts,s^\opts)}u_iv_i^\transp,-w^{(t^\opts,s^\opts)})$ with $(y^{(t^\opts,s^\opts)},w^{(t^\opts,s^\opts)})$ given by \cref{eq:tilde_var_y_thm,eq:tilde_w_thm}}.
	\end{algorithmic}
	\label{alg:proj_epi}
\end{algorithm}
\subsubsection{Low-rank inducing Frobenius norm}
\begin{prop}
	\label{prop:ell_2_explicit}
	Let $\tau >0$, $g = \tau \ell_2$, $z_v \in \mathbb{R}$, $\gamma = 1$ and $\tilde{z} \in \mathbb{R}_{\geq 0}^{r-t+1}$ with $\tilde{z} = \sort(\tilde{z})$. If $f(w) = \frac{1}{2}(w+z_v)^2$, then 
	\cref{test:in_epi_thm,test:zero_sol_thm} become
	\begin{subequations}
		\begin{align}
		& (\tilde{y},\tilde{w}) = (\tilde{z},zv) ~\ \Longleftrightarrow ~\ \sqrt{\sum_{i=1}^{r-t} \tilde{z}_i^2 + \frac{t}{s+t} \tilde{z}_{r-t+1}^2} \leq -\tau z_v,\label{eq:ellip} \\
		& (\tilde{y},\tilde{w}) = 0 ~\ \Longleftrightarrow ~\ \sqrt{\sum_{i=1}^{r-t} \tilde{y}_i^2 + \frac{s+t}{t} \tilde{y}_{r-t+1}^2} \leq \frac{z_v}{\tau} \label{eq:ellip_polar}
		\end{align} 
		and \cref{test:boundary_epi_thm} derives as
		\begin{align}
		\tilde{y}_i &= \frac{\tilde{z}_i}{1+\frac{\mu}{\tau^2 \tilde{w}}}, \ 1\leq i \leq r-t \label{eq:sol_ell_2_tilde1}\\
		\tilde{y}_{r-t+1} &= \frac{ \tilde{z}_{r-t+1}}{1+\frac{\mu t}{\tau^2 \tilde{w} (s+t)}}\label{eq:sol_ell_2_tilde2}\\
		\tilde{w} &= \mu - z_v \label{eq:sol_ell2_w}  
		\end{align}
		where the unique $\mu \geq 0$ is a solution to the fourth order polynomial \begin{equation}
		\left[\left(\tilde{w}  \tau+\frac{\mu}{\tau}\right)^2-c_1\right]\left[(t + s)\tau \tilde{w}  + \frac{\mu}{\tau} t\right]^2 - t c_2^2\left(\tilde{w}  \tau+\frac{\mu}{\tau}\right)^2 = 0 \label{eq:poly_proj}
		\end{equation}
		$c_1 := \sum_{i=1}^{r-t} \tilde{z}^2_i$ and $c_2 := \sqrt{t+s}\tilde{z}_{r-t+1}$.
		
		Moreover, if $f(w) = \chi_{[0,\gamma]}(w)$, then \cref{test:in_epi_thm,test:zero_sol_thm} reduce to \cref{eq:ellip} with $z_v = -1$, as well as \cref{test:boundary_epi_thm} is determined by \cref{eq:sol_ell_2_tilde1,eq:sol_ell_2_tilde2,eq:poly_proj} with $\tilde{w} = 1$.
	\end{subequations}	
\end{prop}
A proof to \cref{prop:ell_2_explicit} can be found in \Cref{proof:ell_2_explicit}.

\subsubsection{Low-rank inducing spectral norm}
\label{subsec:ell_inf_derv}
\begin{prop}
	\label{prop:ell_1_explicit}
	Let $\tau > 0$, $g = \tau \ell_\infty$, $z_v \in \mathbb{R}$, $\gamma = 1$ and $\tilde{z} \in \mathbb{R}_{\geq 0}^{r-t+1}$ with $\tilde{z} = \sort(\tilde{z})$. Further, let \begin{align*}
	\hat{z} &:=\left(\tilde{z}_1,\ldots,\tilde{z}_j,\frac{t}{\sqrt{(t+s)}}\tilde{z}_{r-t+1},\tilde{z}_{j+1},\ldots,\tilde{z}_{r-t}\right) \in \mathbb{R}^{r-t+1},\\
		\alpha &:= \left(\underbrace{1,\ldots,1}_{\text{length } j},\frac{t^2}{(t+s)},1,\dots,1\right)  \in \mathbb{R}^{r-t+1}.
	\end{align*}
	where $j$ is chosen such that 
	\begin{equation}
	\tilde{z}_j>\tfrac{\sqrt{(t+s)}}{t}\tilde{z}_{r-t+1}\geq \tilde{z}_{j+1} \quad \text{or} \quad \tilde{z}_{r-t}  \geq \tfrac{\sqrt{(t+s)}}{t}\tilde{z}_{r-t+1} . \label{eq:break_point_j}
	\end{equation}
	If $f(w) = \frac{1}{2}(w+z_v)^2$, then 
	\cref{test:in_epi_thm,test:zero_sol_thm} become
	\begin{subequations}
		\begin{align}
		&(\tilde{y},\tilde{w}) = (\tilde{z},zv) ~\ \Longleftrightarrow ~\ \sum_{i=1}^{r-t} |\tilde{z}_i| + \frac{t}{\sqrt{t+s}} |\tilde{z}_{r-t+1}| \leq -\tau z_v \label{eq:half}\\
 &(\tilde{y},\tilde{w}) = 0 ~\ \Longleftrightarrow ~\  \max \left(\tilde{z}_1,\dfrac{\sqrt{t+s}}{t} \tilde{z}_{r-t+1} \right) \leq  \frac{z_v}{\tau} \label{eq:ineq_ell_1_polar}
		\end{align} 
	and \cref{test:boundary_epi_thm} derives as
	\begin{align}
	\tilde{y}_i &= \max\left(\tilde{z}_i - \tfrac{\mu}{\tau},0 \right), \ 1\leq i \leq r-t, \label{eq:l1_tilde1}\\
	\tilde{y}_{r-t+1} &= \max \left(\tilde{z}_{r-t+1} - \tfrac{t \mu}{\sqrt{(t + s)}\tau},0 \right), \label{eq:l1_tilde2}\\
	\tilde{w} &= \mu - z_v \label{eq:l1_w}
	\end{align}
	where $\mu = \hat{\mu}_{k^\opts}$ with $\hat{\mu}_k= \frac{z_v+\sum_{i=1}^{k}\hat{z}_{i}}{1+\sum_{i=1}^k\alpha_i}$ and $k^\opts$ can be identified by a search over $k$ with the following rules for increasing/decreasing $k$:
	\begin{enumerate}[label = \Roman*.]
		\item $k^\opts=\max\{k : \hat{z}_k-\alpha_k\hat{\mu}_k\geq 0\}$
		\item $\hat{z}_k-\alpha_k\hat{\mu}_k\geq 0$ for all $k \leq k^\opts$ 
		\item $\hat{z}_k-\alpha_k\hat{\mu}_k< 0$ for all $k> k^\opts$
	\end{enumerate}
	\end{subequations}
Moreover, if $f(w) = \chi_{[0,\gamma]}(w)$, then \cref{test:in_epi_thm,test:zero_sol_thm} reduce to \cref{eq:half} with $z_v = -1$, as well as \cref{test:boundary_epi_thm} is determined by \cref{eq:l1_tilde1,eq:l1_tilde2}, where $\mu = \hat{\mu}_{k^\opts}$ can be found with the search rules from above and $\hat{\mu}_k = \frac{\sum_{i=1}^{k}\hat{z}_{i}}{\sum_{i=1}^k\alpha_i}$.
\end{prop}
\cref{prop:ell_1_explicit} is prove in \Cref{proof:ell_1_explicit}.
\subsection{Computational Complexity}
\label{subsec:compl}
In the following, we evaluate the computational complexity, i.e. counting all flops (see~\cite{trefethen1997numerical}) of the discussed approaches for computing
$$\prox_{\chi_{\epi(\|\cdot\|_{g,r\ast})}}(Z,z_v) = (Z,z_v) - \proj_{-\epi(\|\cdot\|_{g^D,r})}(Z,z_v).$$
Since the same analysis also applies to the other cases discussed in \cref{eq:important_cases}, this will allow 
us to compare our approach to existing methods. 
Our evaluation starts with a discussion of \cref{alg:proj_epi} for a general gauge function, followed by an explicit discussion for the cases of $g=\ell_2$ and $g = \ell_{\infty}$ in \Cref{subsec:complex_ell_2,subsec:complex_ell_infty}. Assume that the cost for determining $$(y^{(t,s)}_{r-t},y^{(t,s)}_{r-t+1},y^{(t,s)}_{r+s},y^{(t,s)}_{r+s+1})$$ is bounded by $C(n,r)$. Then the complexity of \cref{alg:proj_epi} is the sum of:
\begin{enumerate}
	\item SVD for $Z$ providing all $\sigma_i(Z)$ and $u_iv_i^\transp$ such that $Z = \sum_{i = 1}^n \sigma_i(Z)u_iv_i^\transp$ (see~\cite{trefethen1997numerical}): $\mathcal{O}(n^3)$
	\item Binary search rules in~\cref{lem:rules_t,lem:rules_s} for $t$ and $s$ (see~\cite{knuth1998art}): $$\mathcal{O}(C(n,r) \log(r)\log(n-r))$$ 
\end{enumerate}
By the coordinate transformation \cref{eq:tilde_var_y_thm}, it holds that \begin{align*}
(y^{(t,s)}_{r-t},y^{(t,s)}_{r-t+1},y^{(t,s)}_{r+s},y^{(t,s)}_{r+s+1})= (\tilde{y}_{r-t},\frac{1}{\sqrt{s+t}}\tilde{y}_{r-t+1},\frac{1}{\sqrt{s+t}}\tilde{y}_{r-t+1},\sigma_{r+s+1}(Z))
\end{align*}
and therefore the cost for $C(n,r)$ is determined by the cost $\tilde{C}(n,r)$ for finding $(\tilde{y}_{r-t},\tilde{y}_{r-t+1})$ in \cref{test:in_epi_thm,test:zero_sol_thm,test:boundary_epi_thm}.
In addition, we are  required to compute the full solution $y^{(t^\opts,s^\opts)}$ once an optimal pair $(t^\opts,s^\opts)$ is found. The cost for these pre- and post-computing steps is at most $\mathcal{O}(n)$ and therefore the overall computational complexity of \cref{alg:proj_epi} is bounded by
\begin{equation}
\underbrace{\mathcal{O}(n^3)}_{\textnormal{SVD}} ~+~ \underbrace{\mathcal{O}(\tilde{C}(n,r) \log(r)\log(n-r+1))}_{\textnormal{bineary search for } (t^\opts,s^\opts)} ~+~ \underbrace{\mathcal{O}(n)}_{\textnormal{full solution}~\cref{eq:tilde_var_y_thm}}. \label{eq:complexity_alg1}
\end{equation} 
\begin{rem}
	\label{rem:sum}
	The cost for computing $\tilde{z}_{r-t+1}$ is given by the cost for knowing $\sum_{i=r+1}^{r+s} \sigma_i(Z)$ (for $s > 0$) and $\sum_{i=r-t+1}^{r} \sigma_i(Z)$. Both sums could be computed a priori for all $t$ and $s$ through incremental summation with cost $\mathcal{O}(n)$. However, in practice it may be cheaper to store and re-use the intermediate sums, when deriving $\sum_{i=r-t+1}^{r} z_i$ and $\sum_{i=r+1}^{r+s} z_i$. This means we only need to compute additional intermediate sums whenever $t$ and $s$ get increased within in the binary search. 
\end{rem}
Finally, computing $\prox_{\chi_{\epi(\|\cdot\|_{g,r\ast})}}(Z,z_v)$ from $\proj_{-\epi(\|\cdot\|_{g^D,r})}(Z,z_v)$ only contributes an additional $n+1$ subtractions, leaving this complexity analysis invariant. 
 
\subsubsection{Low-rank inducing Frobenius norms}
\label{subsec:complex_ell_2}
In order to determine the computational cost $\tilde{C}(n,r)$ for $g = \ell_2$, we need to distinguish between
 \cref{eq:ellip,eq:ellip_polar} and the case when $\mu$ has to be determined. Both cases require $\sum_{i=1}^{r-t} \tilde{z}_i^2 = \sum_{i=1}^{r-t} \sigma_i^2(Z)$ as either part of the inequalities or as coefficients in the polynomial \cref{eq:poly_proj}. These sums can be computed once for all $t \in \{1,\dots,r\}$ with cost $\mathcal{O}(r)$. Then testing \cref{eq:ellip,eq:ellip_polar} as well as solving the fourth order polynomial \cref{eq:poly_proj} are of cost $\mathcal{O}(1)$. Thus, the complexity of $\tilde{C}(n,r)$ can be summarized as $\mathcal{O}(1)$. 

Using \cref{eq:complexity_alg1}, the complexity for computing $\prox_{\chi_{\epi(\|\cdot\|_{\ell_2,r\ast})}}(Z,z_v)$ is
\begin{equation}
\underbrace{\mathcal{O}(n^3)}_{\textnormal{SVD}} ~+~ \underbrace{\mathcal{O}(\log(r)\log(n-r+1))}_{\textnormal{bineary search for } (t^\opts,s^\opts)}  ~+~ 
\underbrace{\mathcal{O}(n)}_{\textnormal{full solution}~\cref{eq:tilde_var_y_thm}}~ + ~ \underbrace{O(r)}_{\forall t: \ \sum_{i=1}^{r-t} \tilde{z}_i^2} \label{eq:complexity_ell_2}
\end{equation} 
This is same cost as derived in~\cite{eriksson2015k,lai2014efficient} for determining $\prox_{\frac{\gamma}{2}\|\cdot\|_{\ell_2,r\ast}^2}(Z)$.

\subsubsection{Low-rank inducing spectral norms}
\label{subsec:complex_ell_infty}
As in the previous case, in order to compute $\tilde{C}(n,r)$ for $g =  \ell_\infty$, we distinguish between \cref{eq:half,eq:ineq_ell_1_polar} and the case when $\mu$ has to be determined. \cref{eq:half,eq:ineq_ell_1_polar} require $\sum_{i=1}^{r-t}\tilde{z}_i = \sum_{i=1}^{r-t}\sigma_i(Z)$. This can be done once for all $t \in \{1,\dots,r\}$ with cost $\mathcal{O}(r)$. Verifying the corresponding inequalities is then of complexity $\mathcal{O}(1)$. 

For determining $\mu$ we need to 
\begin{enumerate}[label = \alph*)]
	\item Find $j$ in \cref{eq:break_point_j}: $\mathcal{O}(\log(r-t+1))$, because $\tilde{z}_1 \geq \dots \geq \tilde{z}_{r-t}$.
	\item Determine $\mu_{k^{\opts}} = \mu$ through binary search: $\mathcal{O}(r-t+1)$, because $\sum_{i=1}^{r-t+1} \hat{z}_i$ may need to be computed. 
\end{enumerate}
Thus, $\tilde{C}(n,r)$ is given by the complexity of determining $\mu$, which by the preceding analysis is of at most $\mathcal{O}(r)$ and therefore the complexity for computing $\prox_{\chi_{\epi(\|\cdot\|_{\ell_\infty,r\ast})}}(Z,z_v)$ is
\begin{equation}
\underbrace{\mathcal{O}(n^3)}_{\textnormal{SVD}} ~+~ \underbrace{\mathcal{O}(r\log(r)\log(n-r+1))}_{\textnormal{bineary search for } (t^\opts,s^\opts)} ~+~ 
\underbrace{\mathcal{O}(n)}_{\textnormal{full solution}~\cref{eq:tilde_var_y_thm}} ~ + ~ \underbrace{O(r)}_{\forall t: \ \sum_{i=1}^{r-t}\tilde{z}_i} \label{eq:complexity_ell_inf}
\end{equation}
Compared to \cite{wu2014moreau}, our approaches reduces the cost for finding $(t^\opts,s^\opts)$, significantly, from $\mathcal{O}(r(n-r+1) )$ to $\mathcal{O}(r\log(r)\log(n-r+1)+n)$. This is especially important for the corresponding vector-valued problem, when rank is replaced by cardinality.

\section{Case Study: Matrix Completion}
\label{sec:example}
In the following, we will see how the binary search parameters $(t,s,k)$ from \cref{alg:proj_epi} and \cref{prop:ell_1_explicit} evolve when solving an optimization problem with proximal splitting. We consider the convexified low-rank matrix completion problem (see, e.g., \cite{candes2009exact,chandrasekaran2012convex,grussler2016lowrank} for motivation and examples)
\begin{equation}
\label{prob:comp}
\begin{aligned}
& \underset{M}{\textnormal{minimize}}
& &  \|M\|_{\ell_\infty,r\ast} \\
&  \textnormal{subject to} & & m_{ij} = n_{ij}, \ (i,j) \in \mathcal{I}
\end{aligned}
\end{equation} 
with  $r =50$, $\mathcal{I} := \{n_{ij}: n_{ij} > 0 \}$ and $N = \sum_{i=1}^r u_i u_i^\transp $ being defined through the SVD of
\begin{equation}
\label{eq:Hankel}
H :=\begin{tikzpicture}[baseline=(current bounding box.center)]
\matrix (m) [matrix of math nodes,
nodes in empty cells,
right delimiter={)},
left delimiter={(}]{
	1  	& 	1 	& 	  	&   	& 	1 	& 	1  \\
	1  	& 		& 		& 		&  	    & 	0	\\
	&		&		&		&		&		\\
	&		&		&		&		& 		\\
	1	&		&		&		&		& 	0  	\\
	1   &  	0   &		&		&	0	&	0	\\
} ;
\newdimen\L
\L = .8 pt
\draw[loosely dotted, line width = \L] (m-1-2)-- (m-1-5);
\draw[loosely dotted, line width = \L] (m-6-2)-- (m-6-5);
%
\draw[loosely dotted, line width = \L] (m-2-1)-- (m-5-1);
\draw[loosely dotted, line width = \L] (m-2-6)-- (m-5-6);
%
\draw[loosely dotted, line width = \L] (m-5-1)-- (m-1-5);
\draw[loosely dotted, line width = \L] (m-6-1)-- (m-1-6);
\draw[loosely dotted, line width = \L] (m-6-2)-- (m-2-6);


\end{tikzpicture} = \sum_{i=1}^{500} \sigma_i(H)u_i u_i^\transp \in \mathbb{R}^{500 \times 500}.
\end{equation} 
Note that a smaller version of this example has been solved successfully in \cite{grussler2016lowrank} by using SDP-solvers, but this larger example is far out of the scope  of typical SDP-solvers \cite{peaucelle2002user,toh2004implementation}. Therefore, we are going to apply the following Douglas-Rachford splitting scheme (see~\cite{douglas1956numerical,combettes2011proximal,lions1979splitting}):
\begin{equation}
\label{eq:DR}
\begin{aligned}
& X_{i} &=&  &&\prox_{{\normAast{\cdot}{\ell_{\infty},r}}}(Z_{i-1}) \\
&  Y_{i}&= & &&\proj_{\mathcal{L}}(2X_{i} - Z_{i-1})\\
& Z_{i} &= &  &&Z_{i} + Y_{i}-X_{i}
\end{aligned}
\end{equation} 
with $\mathcal{L} := \{X \in \mathbb{R}^{500\times 500}: x_{ij} = n_{ij}, \ (i,j) \in \mathcal{I}  \}$, $Z_0 = 0$ and $\lim_{i \to \infty} X_i = \lim_{i \to \infty} Y_i$ being a solution to \cref{prob:comp}. By the construction of $N$, it can be shown that $\lim_{i \to \infty} X_i = N$ (see~\cite{grussler2016lowrank}). 
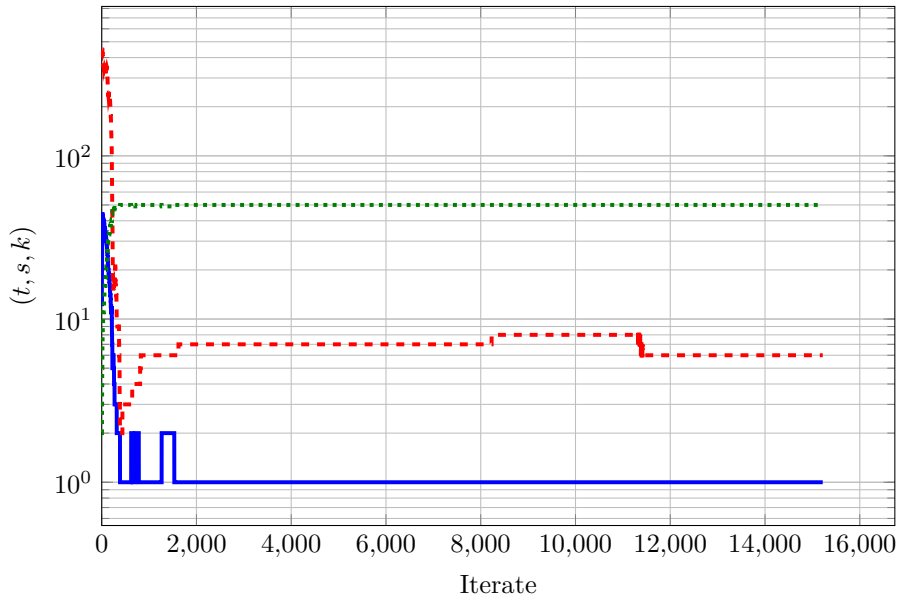
\begin{figure}[tb]
	\centering
\begin{tikzpicture}
\begin{axis}[xmin = 1,
ymode = log,
xlabel= Iterate,
ylabel = \textnormal{$(t,s,k)$},
grid = both,
width = \linewidth,
height = 0.7 \linewidth,
tick label style={/pgf/number format/fixed},
scaled ticks=false
]
\addplot[color = FigColor1] file{ex_path_t.txt};
\label{line:path_t}
\addplot[color = FigColor2 , dashed] file{ex_path_s.txt};
\label{line:path_s}
\addplot[color = FigColor3 ,dotted] file{ex_path_k.txt};
\label{line:path_k}
\end{axis}
\end{tikzpicture}
	\caption{Parameter path of $(\ref{line:path_t}t, \ref{line:path_s}s, \ref{line:path_k} k)$ from \cref{alg:proj_epi} and \cref{prop:ell_1_explicit} when computing $\prox_{{\normAast{\cdot}{\ell_{\infty},r}}}$ within the Douglas-Rachford iterations \cref{eq:DR}. There are no values for the first iterate, because $\prox_{{\normAast{\cdot}{\ell_{\infty},r}}}(Z_{0}) = 0$ and the iterations are stopped when $\|X_i-Y_i\|_F \leq 10^{-8}$. The local plateauing after relatively few iterations suggests to use $(t,s,k)$ from the previous iterations as an initial guess for the current iteration to save computational time.  \label{fig: ex_path}} 
\end{figure}
The parameter path of $(t,s,k)$ for computing $X_i$ is shown in \Cref{fig: ex_path}. We observe that as $X_i$ approaches $N$, the values of $t$, $s$ and $k$ start plateauing. Thus by using the values from one iterate in the subsequent iterate, the practical computational cost may reduce significantly. Finally, after the initial transient, the variance of each parameter is small compared to the overall 500 singular values. As a result, it may be worth considering sparse SVD algorithms, which only computes a small predefined number of largest singular values (see~e.g. \cite{liu2013limited}).

\section{Conclusion}
This work presents a binary search framework for computing the proximal mappings of all unitarily invariant low-rank inducing norms and their epigraph projections. In particular, complete algorithms for the low-rank inducing Frobenius and spectral norms are presented. Our framework unifies and extends the known proximal mapping computations in the following sense: (i) So far, only proximal mappings for the squared low-rank inducing Frobenius norm \cite{eriksson2015k} and the (non-squared) low-rank inducing spectral norm \cite{wu2014moreau} have been derived. This framework is independent of the particular unitary invariant norm and its composition with a convex increasing function. (ii) Excluding the cost for an SVD, we recover the same complexity $\mathcal{O}(n \log n)$ for the squared low-rank inducing Frobenius norm as in~\cite{eriksson2015k,lai2014efficient}, but decrease the complexity for the (non-squared) low-rank inducing spectral norm from $\mathcal{O}(r(n-r+1) )$ in \cite{wu2014moreau} to $\mathcal{O}(r \log(r) \log(n-r+1)+n)$. 

Finally, in our case study we have seen that within a proximal splitting method, this cost may be reduced to $\mathcal{O}(n)$ after a small number of iterations and is therefore roughly the same as in case of the nuclear norm. Implementations for the low-rank inducing Frobenius and spectral norms are available for MATLAB and Python at \cite{grussler2018github,grussler2018githubpy}.

\bibliographystyle{abbrv}
\bibliography{refopt,refpos}

\appendix
	\section{Appendix}
	\counterwithin{prop}{section}
	\counterwithin{defn}{section}
	\counterwithin{cor}{section}
	\counterwithin{lem}{section}
	\counterwithin{algorithm}{section}
	\setcounter{prop}{0}
	\setcounter{defn}{0}
	\setcounter{lem}{0}
	\setcounter{cor}{0}
	\setcounter{algorithm}{0}
\subsection{Search Rules}
\begin{lem}
	\label{lem:rules_t}
	Let $f$ be proper, closed and convex, $z_1 \geq \dots \geq z_n\geq 0$ and $\left(y^{(t)},w^{(t)}\right)$ denote the $t$-depended solution to \begin{equation}
	\label{opt:conj_epi_vec_t}
	\begin{aligned}
	& \underset{y,w}{\textnormal{minimize}}
	& & f(w)+ \frac{\gamma}{2}\sum_{i=1}^n(y_i - z_i)^2 \\
	& \textnormal{subject to}
	& & w \geq  g^D_r(y), \ y \in \mathbb{R}^n,\\
	& & & y_{r-t+1} = \dots = y_{r} \geq \dots \geq y_n.
	\end{aligned}
	\end{equation} where $1\leq t \leq r$. Further, let $\left(y^{(t^\opts)},w^{(t^\opts)}\right)$ be the solution to \begin{equation}
	\label{opt:conj_epi_vec}
	\begin{aligned}
	& \underset{y,w}{\textnormal{minimize}}
	& & f(w)+ \frac{\gamma}{2}\sum_{i=1}^n(y_i - z_i)^2 \\
	& \textnormal{subject to}
	& & w \geq g^D_r(y), \ y \in \mathbb{R}^n,\\
	& & & y_1 \geq \dots \geq y_n,
	\end{aligned}
	\end{equation}	
	such that $y^{(t^\opts)}_{r-t^\opts} > y_{r-t^\opts+1}^{(t^\opts)}$ and $y^{(t^\opts)}_{r-t^\opts} = y_{r-t^\opts+1}^{(t^\opts)}$ if $t^\opts = r$. Then,
	\begin{enumerate}[label=\roman*.]
		\item $t^\opts = \min \left \lbrace \lbrace t: y^{(t)}_{r-t} > y_{r-t+1}^{(t)} \rbrace \cup \lbrace r \rbrace \right \rbrace$. \label{item:i}
		\item If $y^{(t')}_{r-t'} \geq y_{r-t'+1}^{(t')}$ then $y^{(t)}_{r-t} \geq y_{r-t+1}^{(t)}$ for all $t \geq t'$. \label{item:ii}
		\item If $y^{(t')}_{r-t'} < y_{r-t'+1}^{(t')}$ then $y^{(t)}_{r-t} < y_{r-t+1}^{(t)}$ for all $t \leq t'$.\label{item:iii}
	\end{enumerate}
	In particular, $t^\opts$ can be found by a search over $t$, where $t$ is increased/decreased according to the following rules:
	\begin{enumerate}[label=\Roman*.]
		\item $y^{(t)}_{r-t} \geq y_{r-t+1}^{(t)}$ for all $t \geq t^\opts$.\label{item:I}
		\item $y^{(t)}_{r-t} \leq y_{r-t+1}^{(t)}$ for all $t < t^\opts$. \label{item:II}
		\item If $t < t^\opts$ and $y^{(t)}_{r-t} = y_{r-t+1}^{(t)}$ then $\left( y^{(t)},w^{(t)} \right)  = \left( y^{(t^\opts)},w^{(t^\opts)} \right) $.\label{item:III}
	\end{enumerate}
\end{lem}
\label{sec:proof_rules_t}
\begin{proof}
Throughout this proof, we let $p(t)$ denote the optimal cost of \cref{opt:conj_epi_vec_t} as a function of $t$. Since adding constraints cannot reduce the optimal cost, $p$ is a nondecreasing function.\\
\\
\Cref{item:i}: 
By the same reasoning that led to \cref{opt:conj_epi_vec_t}, it holds that
\begin{align}
y_1^{(t)} \geq \dots \geq y_{r-t}^{(t)}\ \text{for} \ 1\leq t \leq r \label{eq:sort_r-t}.
\end{align}
Using \cref{eq:sort_r-t}, the set $\lbrace t: y^{(t)}_{r-t} > y_{r-t+1}^{(t)} \rbrace \cup \lbrace r  \rbrace$ contains all $t$ for which the solution of \cref{opt:conj_epi_vec_t} is feasible for \cref{opt:conj_epi_vec}. Since $p$ is nondecreasing and $\left(y^{(t^\opts)},w^{(t^\opts)}\right)$ is unique, the first claim follows. \\
\\
\Cref{item:ii}: The second claim is proven by contradiction. Let $(y^{(t')},w^{(t')})$ be such that $y^{(t')}_{r-t'} \geq y_{r-t'+1}^{(t')}$. Further assume that  $y^{(t'+1)}_{r-t'-1} < y_{r-t'}^{(t'+1)}$. In the following, we construct another solution $(\tilde{y},\tilde{w}) \in \mathbb{R}^{q+1}$ to \cref{opt:conj_epi_vec_t} with $t = t' +1$, which has a cost that is no larger than $p(t' +1)$. However, \cref{opt:conj_epi_vec_t} has a unique solution due to strong convexity of the cost function. This yields the desired contradiction.

The contradicting solution is constructed as a convex combination \linebreak $\tilde{w}=(1-\alpha)w^{(t'+1)}+\alpha w^{(t')}$ with $\alpha\in(0,1]$ and a partially sorted convex combination of $y^{(t')}$ and $y^{(t'+1)}$ with the same $\alpha$. Let $\hat{y}:=(1-\alpha)y^{(t'+1)}+\alpha y^{(t')}$ and let
\begin{align*}
\tilde{y}:=(\textnormal{sort}(\hat{y}_1,\ldots,\hat{y}_{r-t'-2},\hat{y}_{r-t'}),\hat{y}_{r-t'-1},\hat{y}_{r-t'+1},\ldots,\hat{y}_{q}),
\end{align*}
be the partially sorted convex combination.

To select $\alpha$, we note that by assumption, 
\begin{equation*}
y^{(t')}_{r-t'-1}\geq y^{(t')}_{r-t'} \geq y_{r-t'+1}^{(t')} \quad \text{and} \quad y^{(t'+1)}_{r-t'-1} < y_{r-t'}^{(t'+1)}=y_{r-t'+1}^{(t'+1)}.
\end{equation*}
Therefore, there exists an $\alpha \in (0,1]$ such that
\begin{align*}
\tilde{y}_{r-t'}=\hat{y}_{r-t'-1}&=(1-\alpha) y^{(t'+1)}_{r-t'-1} + \alpha y^{(t')}_{r-t'-1}\\
&=  (1-\alpha) y^{(t'+1)}_{r-t'+1} + \alpha y^{(t')}_{r-t'+1}=\hat{y}_{r-t'+1}=\tilde{y}_{r-t'+1}.
\end{align*}
Since 
$$y_{r-t'+1}^{(t')}=\cdots=y_{r}^{(t')} \quad \text{and} \quad y_{r-t'-1}^{(t'+1)}=\cdots=y_{r}^{(t'+1)},$$ it follows that $$\tilde{y}_{r-t'}=\cdots=\tilde{y}_{r}.$$ Furthermore, the construction of $\tilde{y}$ as well as the sorting give that $$\tilde{y}_{r}\geq\cdots\geq\tilde{y}_q \quad \text{and} \quad \tilde{y}_{1}\geq\cdots\geq \tilde{y}_{r-t'-1}.$$ Hence, $\tilde{y}$ satisfies the chain of inequalities in \cref{opt:conj_epi_vec_t} for $t=t'+1$.

It remains to show that $\tilde{y}$ satisfies the epigraph constraint and that the cost is not higher than $p(t'+1)$. These properties are already fulfilled for $\hat{y}$ being a convex combination of two feasible points with costs $p(t')$ and $p(t'+1)$, respectively, where $p(t') \leq p(t'+1)$. Therefore, it is left to show that the sorting involved in $\tilde{y}$ maintains these properties. First, we show that sorting of any sub-vector in $\hat{y}$ does not increase the cost. Suppose that $z_i\geq z_j$, $\hat{y}_i\leq \hat{y}_j$, i.e., $\hat{y}$ is not sorted the same way as $z$. Then
\begin{align*}
\tfrac{1}{2}\left((z_i-\hat{y}_i)^2+(z_j-\hat{y}_j)^2\right) &= (z_i-z_j)(\hat{y}_j-\hat{y}_i)+\tfrac{1}{2}\left((z_i-\hat{y}_j)^2+(z_j-\hat{y}_i)^2\right)\\
&\geq \tfrac{1}{2}\left((z_i-\hat{y}_j)^2+(z_j-\hat{y}_i)^2\right),
\end{align*}
and thus the cost is not increased by sorting $\hat{y}$ or any sub-vector of it. Further, a permutation of the first $r$ elements of $\hat{y}$ does not influence the epigraph constraint, because $g^D_r(\hat{y})$ is permutation invariant by definition.

Next notice that $\tilde{y}$ is obtained from $\hat{y}$ by first swapping $\hat{y}_{r-t'-1}$ and $\hat{y}_{r-t'}$. From the choice of $\alpha$, we conclude that $$\hat{y}_{r-t'}=(1-\alpha)y_{r-t'}^{(t'+1)}+\alpha y_{r-t'}^{(t')}\geq(1-\alpha)y_{r-t'+1}^{(t'+1)}+\alpha y_{r-t'+1}^{(t')}=\hat{y}_{r-t'+1}=\hat{y}_{r-t'-1}.$$ Thus, this swap is a sorting which does neither increase the cost, nor does it violate the epigraph constraint. Analogously, sorting the first $r-t'$ elements of the resulting vector to obtain $\tilde{y}$ has the same effect and therefore we receive the desired contradiction.\\
\\
\Cref{item:iii}: Suppose that there exist $t$ and $t'$ with $t'>t$ such that $y_{r-t'}^{(t')}<y_{r-t'+1}^{(t')}$ and $y_{r-t}^{(t)}\geq y_{r-t+1}^{(t)}$. Then \Cref{item:ii} shows that $y_{r-t'}^{(t')}\geq y_{r-t'+1}^{(t')}$, which is a contradiction.\\
\\
Items~\ref{item:I} to \ref{item:III}: The statements follow immediately from Items~\ref{item:i} to \ref{item:iii}.
\end{proof}
\begin{lem}
	\label{lem:rules_s}
	Let $f$ and $z$ be as in \cref{lem:rules_t} and $\left(y^{(t,s)},w^{(t,s)}\right)$ denote the $s$-depended solution to \begin{equation}
	\label{opt:conj_epi_vec_t_s}
	\begin{aligned}
	& \underset{y,w}{\textnormal{minimize}}
	& & f(w)+ \frac{\gamma}{2}\sum_{i=1}^{n}(y_i - z_i)^2\\
	& \textnormal{subject to}
	& & w \geq g^D_r(y), \ y \in \mathbb{R}^n,\\
	& & & y_{r-t+1} = \dots = y_{r+s}.
	\end{aligned}
	\end{equation} where $0 \leq s \leq r -n$ and $t$ is fixed within $1 \leq t \leq r$. Further, let $\left(y^{(t,s^\opts)},w^{(t,s^\opts)}\right)$ be the solution to \cref{opt:conj_epi_vec_t} such that $y^{(t,s^\opts)}_{r+s^\opts} > y_{r+s^\opts+1}^{(t,s^\opts)}$ and $y^{(t,s^\opts)}_{r+s^\opts} = y_{r+s^\opts+1}^{(t)}$ if $s^\opts = n-r$. Then,
	\begin{enumerate}[label=\roman*.]
		\item $s^\opts = \min \left \lbrace \lbrace s: y^{(t,s^\opts)}_{r+s^\opts} > y_{r+s^\opts+1}^{(t,s^\opts)} \rbrace \cup \lbrace n-r \rbrace \right \rbrace$.
		\item If $y^{(t,s')}_{r+s'} \geq y_{r+s'+1}^{(t,s')}$ then $y^{(t,s)}_{r+s} \geq y_{r+s+1}^{(t,s)}$ for all $s \geq s'$. 
		\item If $y^{(t,s')}_{r+s'} < y_{r+s'+1}^{(t,s')}$ then $y^{(t,s)}_{r+s} < y_{r+s+1}^{(t,s)}$ for all $s \leq s'$.
	\end{enumerate}
	In particular, $s^\opts$ can be found by a search over $s$, where $s$ is increased/decreased according to the following rules:
	\begin{enumerate}[label=\Roman*.]
		\item $y^{(t,s)}_{r+s} \geq y_{r+s+1}^{(t,s)}$ for all $s \geq s^\opts$.
		\item $y^{(t,s)}_{r+s} \leq y_{r+s+1}^{(t,s)}$ for all $s < s^\opts$.
		\item If $s < s^\opts$ and $y^{(t,s)}_{r+s} = y_{r+s+1}^{(t,s)}$ then $\left( y^{(t,s)},w^{(t,s)} \right)  = \left( y^{(t,s^\opts)},w^{(t,s^\opts)} \right)$.
	\end{enumerate}
\end{lem}
The proof of \cref{lem:rules_s} goes analogously to the proof of \cref{lem:rules_t} and is therefore omitted.
\begin{lem}
	\label{lem:coordinate_transf}
	Let $f$ and $z$ be as in \cref{lem:rules_t},  $1 \leq t \leq r$ and $ 0 \leq s \leq n-r$. Moreover, let $\tilde{z} := Tz \in \mathbb{R}^{r-t+1}$ be defined by \cref{eq:transform}
	and be $(\tilde{y}^{(t,s)},w^{(t,s)})$ the $(t,s)$-depended solution to 
	\begin{equation} \label{opt:reduced}
	\begin{aligned}
	& \underset{\tilde{y},w}{\textnormal{minimize}}
	& & f(w)+ \frac{\gamma}{2}\sum_{i=1}^{r-t+1}(\tilde{y}_i - \tilde{z}_i)^2 \\
	& \textnormal{subject to}
	& & w \geq g_{r,s,t}^D(\tilde{y}), \ \tilde{y} \in \mathbb{R}^{r-t+1}
	\end{aligned}
	\end{equation}
	Then $(y^{(t,s)},w^{(t,s)})$ is a solution to \cref{opt:conj_epi_vec_t_s}, where  
	\begin{align}
	y^{(t,s)}_i := \begin{cases}
	\tilde{y}_i^{(t,s)},  &\text{if } 1\leq i \leq r-t,\\
	\dfrac{\tilde{y}_i^{(t,s)}}{\sqrt{t+s}}, 	&\text{if }  r-t+1 \leq i \leq r+s, \\
	z_i, 	&\text{if }  i \geq r+s+1.\\
	\end{cases} \label{eq:tilde_var_y}
	\end{align}
\end{lem}
\begin{proof}
Letting $\tilde{y} \in \mathbb{R}^{r-t+1}$ be defined as
\begin{align}
\tilde{y}_i = \begin{cases}
y_i, \text{ if } 1\leq i \leq r-t,\\
\sqrt{t+s}y_{r-t+1}, 	\text{ if } i = r-t+1,
\end{cases}
\end{align} and notice that
\begin{align*}
\sum_{i = r-t+1}^{r+s} (y_r - z_i)^2 
& = \left(\tilde{y}_{r-t+1}- \tilde{z}_{r-t+1} \right)^2 + \sum_{i = r-t+1}^{r+s} z_i^2 - \left(\frac{1}{\sqrt{t+s}}\sum_{i = r-t+1}^{r+s}  z_i \right)^2,
\end{align*}
yields the reduced dimensional problem \cref{opt:reduced}.
\end{proof}
\begin{lem} 
	\label{lem:sol}
	$(\tilde{y},\tilde{w})$ is a solution to \cref{opt:reduced} if and only if one of the following cases applies: 
	\begin{subequations}
		\label{eq:cases}
		\begin{align}
		&\text{Case 1: }\tilde{y} = \tilde{z} ~\ \Longleftrightarrow ~\ \tilde{w} = \argmin_w f \text{ and } \tilde{w} \geq g^D_{r,s,t}(\tilde{z})\label{test:in_epi}\\
		& \text{Case 2: }  (\tilde{y},\tilde{w}) = 0 ~\ \Longleftrightarrow ~ g_{r,s,t}(\tilde{z}) \leq  \frac{\mu}{\gamma} \text{ and } \mu \in \partial  f(0) \label{test:zero_sol}\\
		& \text{Case 3: } \frac{\gamma}{\mu} (\tilde{z}-\tilde{y}) \in \partial g^D_{r,s,t}(\tilde{y}) ~\  \mu \in \partial f(\tilde{w})\cap \mathbb{R}_{\geq 0} ~\ \text{and} ~\ \tilde{w} = g^D_{r,s,t}(\tilde{y})\label{test:boundary_epi} 
		\end{align}
	\end{subequations}
\end{lem}

\begin{proof}
A solution $(\tilde{y},\tilde{w})$ to \cref{opt:reduced} fulfills by \cite[Theorem~VI.2.2.1]{hiriart2013convex}
\begin{align}
\label{eq:zero_opt}
0 \in \begin{pmatrix}
\gamma (\tilde{y} -\tilde{z})  \\
\partial f(w^\opts)
\end{pmatrix} + \mathcal{N}_{\epi(g_{r,s,t}^D)}(\tilde{y},\tilde{w})
\end{align}
where $\mathcal{N}$ denotes the normal cone to $\epi(g^D_{r,s,t})$ and the summation is understood set-wise. Then by \cite[Proposition~VI.1.3.1]{hiriart2013convex} 
\begin{align}
\label{eq:normal_epi}
\mathcal{N}_{\epi(g^D_{r,s,t})}(\tilde{y},\tilde{w}) = \begin{cases}
\{(\mu G, -\mu): G \in \partial g^D_{r,s,t}(\tilde{y}), \ \mu \geq 0 \} & \text{if } \tilde{w} = g^D_{r,s,t}(\tilde{y})\\
$\{0\}$														   &  \text{if } (\tilde{y},\tilde{w}) \in \inter(\epi(g^D_{r,s,t}))
\end{cases}
\end{align} 
which is why we need to distinguish the cases $\tilde{y} = \tilde{z}$ and $\tilde{w} = g^D_{r,s,t}(\tilde{y})$. Thus the proof follows by invoking \cref{eq:sub_norm}.
\end{proof}

\subsection{Proof to \cref{thm:main}}
\begin{proof}
\label{sec:proof_main}
First notice that $Y^\opts = \sum_{i=1}^n \sigma_i(Y^\opts)u_iv_i^\transp$ with $Z = \sum_{i=1}^n \sigma_i(Z)u_iv_i^\transp$, because cost and constraint in~\cref{opt:dual_prox} are unitarily invariant (see~\cite{watson1992characterization,lewis1995convex}). Consequently, it is equivalent to consider the vector-valued problem \cref{opt:conj_epi_vec} with $z_i = \sigma_i(Z)$ and $y_i = \sigma_i(Y)$ for  $\ 1 \leq i \leq n$. 
\begin{rem}
	It is not necessary to explicitly restrict $y$ to be nonnegative. The unique solution $(y^\opts,w^\opts)$ to \cref{opt:conj_epi_vec} fulfills $0 \leq y_i^\opts \leq z_i$ for $1 \leq i \leq n$. The upper bound holds, because otherwise by \cite[Theorem~7.4.8.4]{horn2012matrix} $$g^D_r(\bar{y}) \leq g^D_r(y^\opts)$$ with $\bar{y}_i^\opts := \min\{z_i,y_i^\opts\}$ and thus $\bar{y}^\opts$ is a feasible solution to \cref{opt:conj_epi_vec} with smaller cost. Similarly, the lower bound holds, because otherwise by \Cref{def:gauge:absolute} $\bar{y}^\opts$ with $\bar{y}_i^\opts = \max\{0,y_i^\opts\}$ is a feasible solution to \cref{opt:conj_epi_vec} with smaller cost \label{rem:bounds_on_solution}
\end{rem}
Then there exists $t^\opts$ such that $y^\opts = \sigma(Y^\opts)$ fulfills 
\begin{align}
y^\opts_{r-t^\opts} > y^\opts_{r-t^\opts+1} = \dots = y^\opts_r, \label{eq:sort_d}
\end{align}
where $t^\opts = r$ if $y^\opts_1 = y^\opts_r$. This assumption implies that  $y_{r-t^\opts} \geq y_{r-t^\opts+1}$ is assumed to be inactive and therefore can be removed from~\cref{opt:conj_epi_vec}. Then also the constraints $$y_1 \geq \dots \geq y_{r-t^\opts}$$ can be removed, because the cost function and the sorting of $z$ ensures that the solution will always fulfill them. This yields the equivalence to \cref{opt:conj_epi_vec_t}. Thus, solving \cref{opt:conj_epi_vec} reduces to finding $t^\opts$ such that \cref{opt:conj_epi_vec_t} solves \cref{opt:conj_epi_vec}. 

In order to solve \cref{opt:conj_epi_vec_t}, one can proceed similarly as with \cref{opt:conj_epi_vec}. There exists $s_t^\opts\geq 0$ such that the solution $(y^{(t)},w^{(t)})$ to \cref{opt:conj_epi_vec_t} satisfies
\begin{align*}
y^{(t)}_{r-t+1} = \dots = y^{(t)}_{r+s_t^\opts} > y^{(t)}_{r+s_t^\opts+1},
\end{align*}
where $s_t^\opts = n-r$ if $y^{(t)}_{r} = y^{(t)}_n$. As before, this allows us to remove the inactive constraint $y_{r+s_t^\opts} \geq y_{r+s_t^\opts+1}$. Then the constraints $y_{r+s_t^\opts+1} \geq \dots \geq y_n$ become redundant, because $y_{j}^{(t)} = z_j, \ j \geq r+s+1$. Therefore, we are left with \cref{opt:conj_epi_vec_t_s}, which by \Cref{lem:coordinate_transf,lem:sol} proves \cref{test:in_epi_thm,test:zero_sol_thm,test:boundary_epi_thm}. The remainder of the theorem is a direct application of \cref{lem:rules_s,lem:rules_t}.
\end{proof}
\subsection{Proof to \cref{prop:ell_2_explicit}}
\label{proof:ell_2_explicit}
\begin{proof}
	For $\tau > 0$ and a gauge function $\tilde{g}$ it holds that  $g = \tau \tilde{g}$ is gauge function with $g^D = \frac{\tilde{g}}{\tau}$. Setting $\gamma =1$ and $f(w) = \frac{1}{2}(w +z_v)$ in \cref{thm:main},  \cref{test:in_epi_thm,test:zero_sol_thm,test:boundary_epi_thm} then become
	\begin{subequations}
		\label{eq:cases_epi}
		\begin{align}
		& (\tilde{y},\tilde{w}) = (\tilde{z},zv) ~\ \Longleftrightarrow ~\ -\tau z_v \geq \tilde{g}^D_{r,s,t}(\tilde{z})\label{test:in_epi_epi}\\
		&(\tilde{y},\tilde{w}) = 0 ~\ \Longleftrightarrow ~\ \tilde{g}_{r,s,t}(\tilde{z}) \leq  \frac{z_v}{\tau} \label{test:zero_sol_epi}\\
		&\frac{\tau}{\mu}(\tilde{z}-\tilde{y}) \in  \partial {\tilde{g}^D_{r,s,t}}(\tilde{y}),~\  \mu = \tilde{w} + z_v \geq 0 ~\ \text{and} ~\ \tau \tilde{w} = \tilde{g}^D_{r,s,t}(\tilde{y}) \label{test:boundary_epi_epi} 
		\end{align}
	\end{subequations}
	For our particular case $\tilde{g} = \ell_2$, it follows immediately that \cref{test:in_epi_epi,test:zero_sol_epi} correspond to \cref{eq:ellip,eq:ellip_polar}. Furthermore, by taking the gradient of $g^D_{r,s,t}$,  \cref{test:boundary_epi_epi} becomes  \cref{eq:sol_ell_2_tilde1,eq:sol_ell2_w,eq:sol_ell_2_tilde2} with the constraints $\mu \geq 0$ and $\tau \tilde{w} = {g^D_{r,s,t}(\tilde{y})}$. Thus it is left to compute $\mu \geq 0$. Plugging \cref{eq:sol_ell_2_tilde1,eq:sol_ell2_w,eq:sol_ell_2_tilde2} into $\tau^2 \tilde{w}^2 = {g^D_{r,s,t}(\tilde{y})}^2$ and making some rearrangements yields
	\begin{align*}
	1 =  \frac{\sum_{i=1}^{r-t}\tilde{z}_i^2}{\left(\tilde{w} \tau+\frac{\mu}{\tau}\right)^2} + \frac{t}{s+t}\frac{\tilde{z}_{r-t+1}^2}{\left(\tilde{w} \tau+\frac{\mu t}{(s+t)\tau}\right)^2}.
	\end{align*}
	Then defining $c_1 := \sum_{i=1}^{r-t} \tilde{z}^2_i$ and $c_2 := \sqrt{t+s}\tilde{z}_{r-t+1}$, this can be rewritten as the fourth order polynomial equation \cref{eq:poly_proj} which can be solved explicitly for unique $\mu \geq 0$ after the substitution \cref{eq:sol_ell2_w} is performed. This proves the first part of \cref{prop:ell_2_explicit}. 
	For $f(w) = \chi_{[0,\gamma]}(w)$,  \cref{test:in_epi_thm,test:zero_sol_thm,test:boundary_epi_thm} are
	\begin{subequations}
		\begin{align}
		& \tilde{y} = \tilde{z} ~\ \Longleftrightarrow ~\ \tau \geq \tilde{g}^D_{r,s,t}(\tilde{z})\label{test:in_epi_norm}\\
		&\frac{\gamma}{\mu}(\tilde{z}-\tilde{y}) \in  \partial \tilde{g}^D_{r,s,t}(\tilde{y}),~\  \mu \geq 0  ~\ \text{and} \ \tau = \tilde{g}^D_{r,s,t}(\tilde{y}) \label{test:boundary_epi_norm} 
		\end{align}
	\end{subequations}	
	Note that \cref{test:zero_sol_thm} is redundant here, because it coincides with \cref{test:in_epi_norm}. Hence, for $g = \ell_2$ \cref{test:in_epi_norm} becomes \cref{eq:ellip} with $z_v = -1$ and \cref{test:boundary_epi_norm} is equivalent to \cref{eq:poly_proj,eq:sol_ell_2_tilde1,eq:sol_ell_2_tilde2} with $\tilde{w} = 1$.
\end{proof}

\subsection{Break Point Search}
\begin{lem}\label{lem:break_point}
	Let $(\tilde{z},z_v)$ fulfill neither of \cref{eq:half,eq:ineq_ell_1_polar}, and $\hat{z}$ and $\alpha$ be as in \cref{prop:ell_1_explicit}. Further, let $\mu^\opts$ be the solution to \begin{align}
	\sum_{i=1}^{r-t+1}\max(\hat{z}_{i}-\alpha_i\mu,0)  +z_v - \mu=0
	\label{eq:mu_eq}
	\end{align} and $\hat{\mu}_k$ be the solution to
	\begin{align}
	\left(\sum_{i=1}^{k} \hat{z}_{i}-\alpha_i\mu\right)+z_v - \mu=0, {\textnormal{\qquad i.e.,\qquad }}\hat{\mu}_k = \frac{z_v+\sum_{i=1}^{k}\hat{z}_{i}}{1+\sum_{i=1}^k\alpha_i}.
	\label{eq:mu_hat_k}
	\end{align}
	Then there exists $k^\opts \in \{1,\dots r-t+1 \}$ such that  
	\begin{align*}
	\hat{z}_{k^\opts}- \alpha_{k^\opts}{\mu^\opts} \geq 0, \qquad \hat{z}_{i}-\alpha_{i}{\mu^\opts} < 0 \ \textnormal{ for all } \ i > k^\opts,
	\end{align*}
	and
	\begin{enumerate}[label = \roman*.]
		\item $\hat{\mu}_{k^\opts} = \mu^\opts.$ \label{item:mu_kopt}
		\item $k^\opts=\max\{k : \hat{z}_k-\alpha_k\hat{\mu}_k\geq 0\}$.\label{item:k_bisect_opt}
		\item If $\hat{z}_k-\alpha_k\hat{\mu}_k\geq 0$, then $\hat{z}_i-\alpha_i\hat{\mu}_i\geq 0$ for all $i \leq k$.\label{item:k_bisect_pos}
		\item If $\hat{z}_k-\alpha_k\hat{\mu}_k< 0$, then $\hat{z}_i-\alpha_i\hat{\mu}_i< 0$ for all $i\geq k$.\label{item:k_bisect_neg}
	\end{enumerate} 
	In particular,
	\begin{enumerate}[label = \Roman*.]
		\item $\hat{z}_k-\alpha_k\hat{\mu}_k\geq 0$ for all $k \leq k^\opts$.\label{item:k_bisect_pos_2}
		\item $\hat{z}_k-\alpha_k\hat{\mu}_k< 0$ for all $k> k^\opts$.\label{item:k_bisect_neg_2}
	\end{enumerate}
\end{lem}
\label{sec:proof_break}
	\begin{proof}
		We first show some results needed to prove Items~\ref{item:k_bisect_opt} and \ref{item:k_bisect_pos}. Let 
		\begin{align*}
		g_k(\mu):=\sum_{i=1}^{k}\max(\hat{z}_{i}-\alpha_i\mu,0)+z_v-\mu,
		\end{align*}
		Let $\mu_k$ be the unique solution to the equation
		\begin{align*}
		g_k(\mu)=0.
		\end{align*}
		Since all $g_i$ are strictly decreasing in $\mu$ and $$g_k(\mu) = g_{k-1}(\mu)+ \max(\hat{z}_{k}-\alpha_k\mu,0) \geq g_{k-1}(\mu),$$ we conclude that
		\begin{enumerate}[label = \alph*.]
			\item \label{item:mu_inc} $\mu_{k-1}\leq \mu_{k}.$
			\item \label{item:mu_equal} $\hat{z}_k-\alpha_k\mu_k\leq 0 \ \Leftrightarrow \ g_{k-1}(\mu_k)=g_k(\mu_k)=0 \ \Leftrightarrow \ \mu_{k-1}=\mu_k$. 
		\end{enumerate} 
		Moreover, the break point sorting in $\hat{z}$ implies that if $l$ and $\mu$ are such that $\hat{z}_l-\alpha_l\mu\geq 0$, then also $\hat{z}_i-\alpha_i\mu\geq 0$ for all $i\leq l$. Thus,
		\begin{equation*}
		\hat{z}_k-\alpha_k\mu\geq 0  \ \Leftrightarrow \ \sum_{i=1}^k\max(\hat{z}_i-\alpha_i\mu,0)+z_v-\mu=\left(\sum_{i=1}^k\hat{z}_i-\alpha_i\mu\right)+z_v-\mu
		\end{equation*}
		In conjunction with the uniqueness of ${\mu}_k$, this implies that
		\begin{enumerate}[label = \alph*.]
			\setcounter{enumi}{2}
			\item \label{item:mu_agree} $\hat{z}_k-\alpha_k\mu_k\geq 0$ or $\hat{z}_k-\alpha_{k}\hat{\mu}_k\geq 0 \ \Leftrightarrow \ \hat{\mu}_k=\mu_k$.
		\end{enumerate}
	\Cref{item:mu_kopt}: This has already been proven in the discussion before \Cref{lem:break_point}.
	\\
		\Cref{item:k_bisect_opt}: By the definition of $k^\opts$ and \Cref{item:mu_kopt} it holds that
		\begin{equation}
		\hat{z}_{k^\opts}-\alpha_{k^\opts}\hat{\mu}_{k^\opts}\geq 0  \qquad \textnormal{and} \qquad \hat{z}_i-\alpha_i\hat{\mu}_{k^\opts} < 0 \ \ \textnormal{for all} \ \ i > k^\opts. \label{eq:ineq_kopt}
		\end{equation}
		Thus by \Cref{item:mu_agree} $\hat{\mu}_{k^{\opts}} = \mu^\opts =  \mu_{k^{\opts}}$ and
		\begin{align*}
		\hat{z}_i-\alpha_i \mu_{i} < 0 \ \ \textnormal{for all} \ \ i > k^\opts.
		\end{align*}
		Then \Cref{item:mu_equal} implies that
		\begin{equation*}
		\hat{\mu}_{k^{\opts}} = \mu^\opts = \mu_{r-t+1} = \mu_{r-t} = \dots = \mu_{k^{\opts}}.
		\end{equation*}
		Therefore, if there exists $k > k^\opts$ with $\hat{z}_k-\alpha_k\hat{\mu}_k\geq 0$, it will hold by \Cref{item:mu_agree} that 
	\begin{equation*}
	\hat{\mu}_k = \mu_k = \hat{\mu}_{k^{\opts}},
	\end{equation*}
	which contradicts \cref{eq:ineq_kopt}, because
	\begin{equation*}
	0 \leq \hat{z}_k-\alpha_k\hat{\mu}_k = \hat{z}_k-\alpha_k\hat{\mu}_{k^\opts} < 0.
	\end{equation*}
	This proves that $k^\opts = \max \{k: \hat{z}_k-\alpha_k\hat{\mu}_k\geq 0\}$.
	\\
		\Cref{item:k_bisect_pos}: Assume that $\hat{z}_{k}-\alpha_{k}\hat{\mu}_k\geq 0$. Then, by the break point sorting it holds that $ \hat{z}_{k-1}-\alpha_{k-1}\hat{\mu}_k\geq 0$ and by \Cref{item:mu_agree,item:mu_inc} that $\hat{\mu}_k = \mu_k \geq \mu_{k-1}$. Thus, we conclude that
		\begin{align*}
		0\leq \hat{z}_{k-1}-\alpha_{k-1}\hat{\mu}_k=\hat{z}_{k-1}-\alpha_{k-1}\mu_k\leq \hat{z}_{k-1}-\alpha_{k-1}\mu_{k-1}=\hat{z}_{k-1}-\alpha_{k-1}\hat{\mu}_{k-1},
		\end{align*}
		where the last equality follows again by  \Cref{item:mu_agree}.
		The other indices follow inductively. 
		\\
		\Cref{item:k_bisect_neg}: Let on the contrary $k$ be such that $\hat{z}_k-\alpha_k\hat{\mu}_k<0$, but with $i\in\{k,\ldots,r-t+1\}$ such that $\hat{z}_i-\alpha_i\hat{\mu}_i\geq 0$. Then, by \Cref{item:k_bisect_pos}, $\hat{z}_k-\alpha_k\hat{\mu}_k\geq 0$, which is a contradiction.
		\\
		Items~\ref{item:k_bisect_pos_2} and \ref{item:k_bisect_neg_2}: Follow immediately from Items~\ref{item:k_bisect_opt}~to~\ref{item:k_bisect_neg}. 
	\end{proof}

\subsection{Proof to \cref{prop:ell_1_explicit}}
\label{proof:ell_1_explicit}
\begin{proof}
	Analogous to showing \cref{prop:ell_2_explicit}, \cref{test:in_epi_epi} and \cref{test:zero_sol_epi} correspond to \cref{test:in_epi_thm,test:zero_sol_thm}, which translate for $\tilde{g} = \ell_{\infty}$ to
	\begin{subequations}
		\begin{align}
		&(\tilde{y}^\opts,w^\opts) = (\tilde{z},zv) ~\ \Longleftrightarrow ~\ \sum_{i=1}^{r-t} |\tilde{z}_i| + \frac{t}{\sqrt{t+s}} |\tilde{z}_{r-t+1}| \leq - \tau z_v \notag \\
		& (\tilde{y}^\opts,w^\opts) = 0 ~\ \Longleftrightarrow ~\  \max\left(|\tilde{z}_1|,\ldots,|\tilde{z}_{r-t-2}|,\dfrac{\sqrt{t+s}}{t}|\tilde{z}_{r-t+1}|\right) \leq \frac{z_v}{\gamma} \notag
		\end{align} 
		Since $\tilde{z}$ is nonnegative and decreasingly sorted, the second case simplifies to \cref{eq:ineq_ell_1_polar}.
		For \cref{test:boundary_epi_epi}, we need to note that $\tilde{y} \in \mathbb{R}^{r-t+1}_{\geq 0}$ and therefore the conditions for $\tilde{y}_i=0$ and $\tilde{y}_i>0$ become
		\begin{align*}
		\tilde{y}_i =0\Leftrightarrow  \tilde{z}_i\in \left[0,\tfrac{\mu}{\tau}\right], \qquad \tilde{y}_i>0\Leftrightarrow \tilde{y}_i=\tilde{z}_i-\tfrac{\mu}{\tau}
		\end{align*}
		for all $i \in \{1,\ldots,r-t \}$. These equivalences also hold for $\tilde{y}_{r-t+1}$ with $\mu$ multiplied by $t/\sqrt{s+t}$. Therefore, \cref{eq:l1_tilde1,eq:l1_tilde2,eq:l1_w} follow together with the constraints $\tau \tilde{w} = \tilde{g}^D_{r,s,t}(\tilde{y})$ and $\mu \geq 0$. 
		Then, plugging \cref{eq:l1_tilde1,eq:l1_tilde2} into $\tau w^\opts = \tilde{g}^D_{r,s,t}(\tilde{y})$ yields
		\begin{align}
		0 &= \frac{1}{\tau}\sum_{i=1}^{r-t} |\tilde{y}_i| + \dfrac{t}{\sqrt{t+s}}|\tilde{y}_{r-t+1}| - \tilde{w} \notag \\ 
		&= \sum_{i=1}^{r-t} \max\left(\frac{\tilde{z}_i}{\gamma} - \frac{\mu}{\gamma^2},0 \right) + \max \left(\dfrac{t}{\sqrt{t + s}\gamma}\tilde{z}_{r-t+1} - \dfrac{t^2 \mu}{(t + s)\gamma^2},0 \right) +  z_v - \mu. \label{eq:max_express}
		\end{align}
		which determines the unique solution to $\mu \geq 0$.
	\end{subequations}
	We solve the equation by using a so-called \emph{break point searching algorithm}, as it has been done for similar problems in~\cite{held1974validation,duchi2008efficient,condat2016fast}. 
	
	In our case, the break points are given by the smallest values of $\mu$ for which each max expressions as function of $\mu$ becomes zero, i.e., 
	\begin{subequations}
		\begin{align*}
		\left( \gamma \tilde{z}_1, \dots, \gamma \tilde{z}_{r-t}, \frac{\gamma \sqrt{s+t}}{t} \tilde{z}_{r-t+1}  \right)
		\end{align*}
		Then we define 
		\begin{align*}
		\hat{z} :=\frac{1}{\gamma}\left(\tilde{z}_1,\ldots,\tilde{z}_j,\dfrac{t}{\sqrt{(t+s)}}\tilde{z}_{r-t+1},\tilde{z}_{j+1},\ldots,\tilde{z}_{r-t}\right),
		\end{align*}
		to be the vector that sorts $\frac{1}{\gamma}\left(\tilde{z}_1,\dots,\tilde{z}_{r-t},\frac{t}{\sqrt{t+s}} \tilde{z}_{r-t+1}\right)$ by decreasing break points, i.e., $j$ fulfills
		\begin{equation}
		\tilde{z}_j>\tfrac{\sqrt{(t+s)}}{t}\tilde{z}_{r-t+1}\geq \tilde{z}_{j+1} \quad \text{or} \quad \tilde{z}_{r-t}  \geq \tfrac{\sqrt{(t+s)}}{t}\tilde{z}_{r-t+1}.\label{eq:break_j}
		\end{equation}
		Therefore, \cref{eq:max_express}  can be equivalently written as 
		\begin{align}
		\sum_{i=1}^{r-t+1}\max(\hat{z}_{i}-\alpha_i\mu,0)  +z_v - \mu=0
		\label{eq:mu_eq}
		\end{align}
		with
		\begin{align*}
		\alpha=\frac{1}{\gamma^2}\left(1,\ldots,1,\dfrac{t^2}{(t+s)},1,\dots,1\right).
		\end{align*}
		Hence, there exists an index $k^\opts \in \{1,\dots,r-t+1\}$ such that the unique solution $\mu \geq 0$ to  \cref{eq:mu_eq}  fulfills 
		\begin{align}
		\hat{z}_{k^\opts}- \alpha_{k^\opts}\mu \geq 0{\textnormal{\qquad and \qquad}} \hat{z}_{i}-\alpha_{i}\mu < 0 \ \textnormal{ for all } \ i > k^\opts,
		\label{eq:mu_opt_cond}
		\end{align}
		which is why $\mu$ can be determined as
		\begin{align}
		\mu=\dfrac{z_v+\sum_{i=1}^{k^\opts}\hat{z}_{i}}{1+\sum_{i=1}^{k^\opts}\alpha_i}
		\label{eq:mu_opts}.
		\end{align}
	\end{subequations}
	Consequently, computing $\mu$ equals a search for $k^\opts \in \{1,\dots,r-t+1\}$ for which \cref{eq:mu_opts} satisfies \cref{eq:mu_opt_cond}. This can be done with the search rules in \cref{lem:break_point}.
	
	Finally, if $f(w) = \chi_{[0,\gamma]}(w)$, then \cref{test:in_epi_thm,test:zero_sol_thm,test:boundary_epi_thm} are given by \cref{test:in_epi_norm,test:boundary_epi_norm}. For $\tilde{g} = \ell_{\infty}$, this corresponds to \cref{eq:half} with $z_v = -1$, and  \cref{eq:l1_tilde1,eq:l1_tilde2} with the constraint that 
	\begin{align}
	\sum_{i=1}^{r-t+1}\max(\hat{z}_{i}-\alpha_i\mu,0)=\tau,
	\end{align}
	respectively. Therefore, 
	\begin{align}
	\hat{\mu}_k = \frac{\sum_{i=1}^{k}\hat{z}_{i}}{\sum_{i=1}^k\alpha_i} \qquad \text{and} \qquad \mu=\dfrac{\sum_{i=1}^{k^\opts}\hat{z}_{i}}{\sum_{i=1}^{k^\opts}\alpha_i}
	\end{align}
	and it is readily seen that $k^\opts$ obeys the same rules as in \cref{lem:break_point}.
\end{proof}

\end{document}